\documentclass{birkjour}

 \newtheorem{thm}{Theorem}[section]
 \newtheorem{cor}[thm]{Corollary}
 \newtheorem{lem}[thm]{Lemma}
 \newtheorem{prop}[thm]{Proposition}
 \theoremstyle{definition}
 \newtheorem{defn}[thm]{Definition}
 \newtheorem{rem}[thm]{Remark}
 \newtheorem{ex}[thm]{Example}
 \numberwithin{equation}{section}

\begin{document}

\title[The Static Maxwell System in Inhomogeneous Media]
{The Static Maxwell System in Three Dimensional Inhomogeneous Media,
 \\ Generalized Non-Euclidean Modification of the System $(R)$ and Fueter's Construction}

\author[D. Bryukhov]{Dmitry Bryukhov}

\address{%
Science City Fryazino, Russia}

\email{bryukhov@mail.ru  \\   https://orcid.org/0000-0002-8977-3282}



\begin{abstract}
This paper extends approach of our joint paper with K\"{a}hler and
recent paper of the author, published in 2021, on problems of the
static Maxwell system in three dimensional inhomogeneous media.
Applied pseudoanalytic function theory developed by Kravchenko et
al. allows to characterize, in particular, meridional and transverse
fields in cylindrically layered media.
Geometric properties of the electric field gradient ($EFG$) tensor
within a wide range of meridional fields allows us to introduce the
concept of $\alpha$-meridional mappings of the first and second kind
depending on the values of a real parameter $\alpha$.
In case $\alpha =1$ tools of the radially holomorphic potential
provide essentially new meridional models in the context of
generalized axially symmetric potential theory (GASPT). Integral
representations of Bessel functions of the first kind of integer
order and the reduced quaternionic  argument  are first established.
In case $\alpha =0$ geometric properties of harmonic meridional
mappings of the second kind are described. Some open problems in
three dimensional inhomogeneous anisotropic media are discussed
using a generalized Riemannian modification of the system $(R)$.

\end{abstract}

\date{Received: date / Accepted: date}

\maketitle

\section{Introduction, Preliminaries, and Notations}

\subsection{Introduction}
Contemporary aspects of geometrical optics and geo-electrostatics
involve a rich variety of analytic models in the context of the
static Maxwell system in three dimensional inhomogeneous isotropic
media described by a variable $C^1$-coefficient $\phi=
\phi(x_0,x_1,x_2)>0$ (see, e.g.,
\cite{BornWolf:2003,Chew,KhmKravOv:2010,Wait:1982,SvetGub,Br:Hefei2020}):
\begin{equation}
  \left\{
\begin{array}{l}
  \mathrm{div} \, ( \phi \vec E ) = 0, \\
  \mathrm{curl}{\ \vec E} = 0,
 \end{array}
\right.
\label{isotropic-electrostatic-Maxwell-system-3}
\end{equation}
where the vector $\vec E =(E_0, E_1, E_2)$ is known as the electric field strength.
 The coefficient $\phi= \phi(x_0,x_1,x_2)$ in the domain of geometrical optics is interpreted as the dielectric
permittivity $\varepsilon = \varepsilon(x_0, x_1, x_2)$, while in
the domain of geo-electrostatics as the electrical conductivity $
\sigma = \sigma(x_0, x_1, x_2)$.

 The space $ \mathbb R^3=\{(x_0, x_1, x_2)\}$ in our setting includes the longitudinal variable $x_0$.
 The vector $\vec E$ in simply connected open domains $\Lambda \subset \mathbb R^3$
 satisfies the relation $ \vec E = \mathrm{grad} \ h$ (up to the sign of the scalar function $h$).
 The electrostatic potential $ h= h(x_0,x_1,x_2)$ allows us to reduce $C^1$-solutions of the system $(\ref{isotropic-electrostatic-Maxwell-system-3})$
to $C^2$-solutions of the continuity equation:
 \begin{equation}
 \mathrm{div} \, (\phi \ \mathrm{grad}{\ h}) = 0.
\label{isotropic-Lap-Bel-eq-3}
  \end{equation}

The static Maxwell system
$(\ref{isotropic-electrostatic-Maxwell-system-3})$ may be written as
 \begin{equation}
 \left\{
     \begin{array}{l}
       \phi \ \mathrm{div}\ { \vec E}
+ \frac{\partial{\phi}}{\partial{x_0}}E_0 + \frac{\partial{\phi}}{\partial{x_1}}E_1 + \frac{\partial{\phi}}{\partial{x_2}}E_2 =0,  \\
      \frac{\partial{E_0}}{\partial{x_1}}= \frac{\partial{E_1}}{\partial{x_0}}, \ \ \ \ \
      \frac{\partial{E_0}}{\partial{x_2}}= \frac{\partial{E_2}}{\partial{x_0}}, \\
      \frac{\partial{E_1}}{\partial{x_2}}=\frac{\partial{E_2}}{\partial{x_1}},
     \end{array}
  \right.
 \label{Bryukhov-isotropic-Maxwell-system}
\end{equation}
and the continuity equation $(\ref{isotropic-Lap-Bel-eq-3})$,
respectively, may be written as (see, e.g., \cite{Wait:1982})
\begin{equation}
\phi \left( \frac{{\partial}^2{h}}{{\partial{x_0}}^2} + \frac{{\partial}^2{h}}{{\partial{x_1}}^2} + \frac{{\partial}^2{h}}{{\partial{x_2}}^2} \right)
+   \frac{\partial{\phi}}{\partial{x_0}} \frac{\partial{h}}{\partial{x_0}} +
 \frac{\partial{\phi}}{\partial{x_1}} \frac{\partial{h}}{\partial{x_1}} +
  \frac{\partial{\phi}}{\partial{x_2}} \frac{\partial{h}}{\partial{x_2}} =0.
\label{Lap-Bel-eq-3-general}
\end{equation}

 The equation
\begin{equation}
 h(x_0, x_1, x_2) = C = const
\label{equipotential}
\end{equation}
allows us to establish important properties of the equipotential
surfaces in simply connected open domains $\Lambda \subset \mathbb
R^3$. Using the total differential $dh$, the Eq.
$(\ref{equipotential})$ is reformulated as
\begin{equation}
 dh  = \frac{\partial{h}}{\partial{x_0}} d{x_0} + \frac{\partial{h}}{\partial{x_1}} d{x_1}
+ \frac{\partial{h}}{\partial{x_2}} d{x_2} = 0.
\end{equation}

\begin{defn}
Let $\varsigma$ be a real independent variable. Assume that
homogeneous first-order partial differential equation
\begin{equation}
  \frac{\partial{h}}{\partial{x_0}} W_0 + \frac{\partial{h}}{\partial{x_1}} W_1
+ \frac{\partial{h}}{\partial{x_2}} W_2 = 0
\label{PDE}
\end{equation}
 is satisfied in $ \Lambda$, such that
$$
 W_0(x_0, x_1, x_2) = \frac {dx_0}{d\varsigma}, \ \
  W_1(x_0, x_1, x_2) = \frac {dx_1}{d\varsigma}, \ \
 W_2(x_0, x_1, x_2)= \frac {dx_2}{d\varsigma}.
 $$
 The electrostatic potential $ \ h= h(x_0, x_1, x_2)$  is
called the first integral for the characteristic vector field $ \vec
W = (W_0, W_1, W_2)$ in  $ \Lambda$ (see, e.g., \cite{ArnoldGeom}).
\end{defn}

The Eq. $(\ref{PDE})$ is geometrically characterized
 as the orthogonality condition for vector fields $\vec E$ and $\vec W$:
\begin{equation}
  ( \vec E, \vec W ) = (\mathrm{grad} \ h, \vec W ) = 0.
\label{orthogonality-Maxwell-electric}
\end{equation}

  The Eq. $(\ref{orthogonality-Maxwell-electric})$ is satisfied, in particular,
  under condition of $ \vec E = \mathrm{grad} \ h = (u_0,
-u_1, -u_2) = 0$.

\begin{defn}
Let $\Lambda \subset \mathbb R^3$ be a simply connected open domain.
Every point $x^* \in \Lambda$ under condition of $ \mathrm{grad} \
h(x^*) =0$ is called a critical point of the electrostatic potential
$h = h(x)$  in $\Lambda$. The set of critical points is called the
critical set of $h(x)$ in $\Lambda$.
 \end{defn}

Inhomogeneous isotropic media, whose properties are constant throughout every plane perpendicular to a fixed direction,
are referred to as layered media (see, e.g., \cite {BornWolf:2003,Chew}).

The main goal of this paper is to compare applications of two
families of generalizations of the Cauchy-Riemann system with
variable coefficients, in accordance with the static Maxwell system
 in special planarly layered media, where $\phi= \phi(x_2^{-\alpha})$, and in accordance
with the static Maxwell system  in special cylindrically layered
media, where $\phi= \phi(\rho^{-\alpha})$ $(\alpha \in \mathbb{R})$,
respectively.

The paper is organized as follows. In Section 2, we present
$\alpha$-hyperbolic non-Euclidean modification of the system $(R)$
and study new properties of $\alpha$-hyperbolic harmonic potentials
in Cartesian coordinates using Bessel functions of the first and
second kind of real order. New applications of Vekua type systems in
the context of hyperbolic function theory in the plane are
demonstrated. In Section 3, we present $\alpha$-axial-hyperbolic
non-euclidean modification of the system $(R)$ and study new
properties of $\alpha$-axial-hyperbolic harmonic potentials in
cylindrical coordinates using Bessel functions of the first and
second kind of real order. Criterion of joint class of
$\alpha$-hyperbolic harmonic and $\alpha$-axial-hyperbolic harmonic
potentials in Cartesian coordinates is formulated. In Section 4, we
present $(\alpha_1, \alpha_2)$-bi-hyperbolic non-Euclidean
modification of the system $(R)$ in the context of generalized
bi-axially symmetric potential theory.
 Some properties of $(\alpha_1, \alpha_2)$-bi-hyperbolic harmonic potentials
  and $\alpha$-hyperbolic harmonic potentials in Cartesian coordinates are compared.
In Section 5, we focus on the specifics of meridional fields in
cylindrically layered inhomogeneous media. Criterion of joint class
of $\alpha$-hyperbolic harmonic and $\alpha$-axial-hyperbolic
harmonic potentials in cylindrical coordinates is formulated. The
electrostatic potential of every meridional field in special
cylindrically layered media satisfies the elliptic
Euler-Poisson-Darboux equation in cylindrical coordinates. New
concept of $\alpha$-meridional mappings of the first and second
kind, where $\alpha \in \mathbb{R}$, is introduced. In Section 6, in
case $\alpha=1$  the radially holomorphic potential is presented as
an extension of the complex potential in the context of GASPT. A
wide range of meridional electrostatic fields is provided by means
of the reduced quaternionic Fourier-Fueter cosine and sine
transforms of real-valued originals. Applied properties of Bessel
functions of the first kind of integer order and the reduced
quaternionic argument are first demonstrated. In Section 7, in case
$\alpha=0$ geometric properties of harmonic meridional mappings of
the second kind are characterized using Bessel function of the first
kind of order zero. In Section 8, new generalized Riemannian
modification of the system (R) is described into the framework of
problems of the static Maxwell system in three dimensional
inhomogeneous anisotropic media.

\subsection{Preliminaries}

 General class of $C^1$-solutions of the system $(\ref{Bryukhov-isotropic-Maxwell-system})$
  may be equivalently represented as general class of $C^1$-solutions of the system
 \begin{equation}
 \left\{
     \begin{array}{l}
       \phi  \left( \frac{\partial{u_0}}{\partial{x_0}}-
      \frac{\partial{u_1}}{\partial{x_1}}-
      \frac{\partial{u_2}}{\partial{x_2}}\right) + \left(\frac{\partial{\phi}}{\partial{x_0}}u_0 - \frac{\partial{\phi}}{\partial{x_1}}u_1 - \frac{\partial{\phi}}{\partial{x_2}}u_2\right) =0,  \\
      \frac{\partial{u_0}}{\partial{x_1}}=-\frac{\partial{u_1}}{\partial{x_0}}, \ \ \ \ \
      \frac{\partial{u_0}}{\partial{x_2}}=-\frac{\partial{u_2}}{\partial{x_0}}, \\
      \frac{\partial{u_1}}{\partial{x_2}}=\frac{\partial{u_2}}{\partial{x_1}},
     \end{array}
  \right.
\label{Bryukhov-Kaehler-3}
\end{equation}
 where $ \vec E = (u_0, -u_1, -u_2)$.
 This system was first constructed by the author jointly with K\"{a}hler at the University of Aveiro, November 2015.

 We have to deal with the Laplace-Beltrami equation
\begin{equation}
 \Delta_{B} \ h := \phi^{-3} \mathrm{div}( \phi \ \mathrm{grad}{\ h}) = 0
\label{phi-conformal-Laplace-Beltrami}
\end{equation}
with respect to the conformal metric  (see, e.g., \cite{Eisenhart:Riem,Ahlfors:1981})
 \begin{equation}
\ ds^2 = \phi^2 (d{x_0}^2 + d{x_1}^2 + d{x_2}^2).
\label{Riemannian conformal metric}
\end{equation}

 We have to deal with Euclidean geometry in case $\phi=const$.
In particular, some new properties of analytic solutions of the
static Maxwell system in three dimensional homogeneous media
\begin{equation}
  \left\{
\begin{array}{l}
  \mathrm{div} \, { \vec E} = 0, \\
  \mathrm{curl}{\ \vec E} = 0,
 \end{array}
\right.
\label{Riesz-Maxwell-system-3}
\end{equation}
where  $ \ (E_0, E_1, E_2) := (u_0, -u_1, -u_2)$,  have been studied
in the context of quaternionic analysis in $\mathbb R^3$ by Brackx, Delange, Sommen et al.
by means of the reduced quaternion-valued monogenic functions $u = u_0 + iu_1 + ju_2$
 whose components $u_l = u_l(x_0, x_1, x_2)$ $(l = 0, 1, 2)$ are harmonic functions of real variables $x_0, x_1, x_2$
(see, e.g., \cite{BraDelSom:1982,Leut:2000,Del:2007,GM:2011}).

  The electrostatic potential $h =h(x_0, x_1, x_2)$ in homogeneous media satisfies the Laplace equation:
$$
\mathrm{div} \, (\mathrm{grad}{\ h}) = \Delta{h} = 0.
$$

  General class of analytic solutions of the system $(\ref{Riesz-Maxwell-system-3})$
  is equivalently represented as general class of analytic solutions of the system
$$
(R)  \left\{
\begin{array}{l}
  \frac{\partial{u_0}}{\partial{x_0}} -
  \frac{\partial{u_1}}{\partial{x_1}} -
  \frac{\partial{u_2}}{\partial{x_2}} = 0, \\
 \frac{\partial{u_0}}{\partial{x_1}}= - \frac{\partial{u_1}}{\partial{x_0}},
 \ \ \ \frac{\partial{u_0}}{\partial{x_2}}= - \frac{\partial{u_2}}{\partial{x_0}}, \\
 \frac{\partial{u_1}}{\partial{x_2}}= \frac{\partial{u_2}}{\partial{x_1}}.
 \end{array}
\right.
$$
This system is called the system $(R)$ in honor of Riesz (see, e.g., \cite{Leut:2000,Del:2007,GM:2011,MorAvetGuer:2013}).

As noted by several authors, the theory of monogenic functions  in the context of quaternionic analysis in $\mathbb R^3$
 (see, e.g., \cite{BraDelSom:1982,Leut:2000,Del:2007,GM:2009-CUBO,GM:2011})
 does not cover the set of three dimensional M\"{o}bius transformations
(see, e.g., \cite{Ahlfors:1981,GM:2009,GM:2010}). The reduced
quaternionic power functions $u= u_0 + iu_1 + ju_2 = (x_0 + ix_1 +
jx_2)^n$ $(n \in \mathbb{Z})$ are not included into the theory of
the reduced quaternion-valued monogenic functions (see, e.g.,
\cite{ErOrelVie:2017}).

 The system $(\ref{Bryukhov-Kaehler-3})$ may be considered as new generalized non-Euclidean modification of the system $(R)$
with respect to the conformal metric $(\ref{Riemannian conformal metric})$.

The Hessian matrix $\mathbf{H_{l m}}(h) =
\frac{\partial^2{h}}{\partial{x_l} \partial{x_m}} \ $ $ ( l, m =
0,1,2)$ of the electrostatic potential $h = h(x_0, x_1, x_2)$ is
interpreted as the electric field gradient ($EFG$) tensor
$\mathbf{J_{l m}}(\vec E) = \frac{\partial{E_l}}{\partial{x_m}} \ $
$ ( l, m = 0,1,2)$.

\begin{defn}
Every point $x \in \Lambda$ under condition of $\det\mathbf{J}(\vec
E(x)) =0$ is called a degenerate point of the $EFG$ tensor
$\mathbf{J}(\vec E(x))$ in $\Lambda$.
 \end{defn}

Properties of the sets of degenerate points of continuously differentiable mappings and the $EFG$ tensors
are of particular interest to the catastrophe theory  (see, e.g., \cite{BorisTar:1979,PosStew,Gilmore:1993}).

 The characteristic equation of the $EFG$ tensor in our setting
\begin{equation}
 \left(
\begin{array}{lll}
  \frac{\partial{E_0}}{\partial{x_0}} & \frac{\partial{E_0}}{\partial{x_1}} & \frac{\partial{E_0}}{\partial{x_2}} \\
 \frac{\partial{E_1}}{\partial{x_0}}  & \frac{\partial{E_1}}{\partial{x_1}}  &  \frac{\partial{E_1}}{\partial{x_2}} \\
 \frac{\partial{E_2}}{\partial{x_0}}  & \frac{\partial{E_2}}{\partial{x_1}}  &  \frac{\partial{E_2}}{\partial{x_2}}
 \end{array}
\right) =
\left(
\begin{array}{lll}
 \ \ \frac{\partial{u_0}}{\partial{x_0}} &  \ \ \frac{\partial{u_0}}{\partial{x_1}} & \ \ \frac{\partial{u_0}}{\partial{x_2}} \\
 -\frac{\partial{u_1}}{\partial{x_0}}  & -\frac{\partial{u_1}}{\partial{x_1}}  &  -\frac{\partial{u_1}}{\partial{x_2}} \\
 -\frac{\partial{u_2}}{\partial{x_0}}  & -\frac{\partial{u_2}}{\partial{x_1}}  &  -\frac{\partial{u_2}}{\partial{x_2}}
 \end{array}
\right)
  \label{Hessian matrix}
\end{equation}
  is expressed as
\begin{equation}
  \lambda^3 - I_{\mathbf{J}(\vec E)} \lambda^2 + II_{\mathbf{J}(\vec E)} \lambda - III_{\mathbf{J}(\vec E)} =
  0,
\label{characteristic lambda}
\end{equation}
  $ \ I_{\mathbf{J}(\vec E)} = \lambda_0 + \lambda_1 + \lambda_2$,
$ \ II_{\mathbf{J}(\vec E)} = \lambda_0 \lambda_1 + \lambda_0
\lambda_2 + \lambda_1 \lambda_2$, $\ III_{\mathbf{J}(\vec E)} =
\lambda_0 \lambda_1 \lambda_2$.

The principal invariants of the $EFG$ tensor $\mathbf{J}(\vec E)$
are given by formulas
\begin{displaymath}
 \left\{
  \begin{array}{l}
I_{\mathbf{J}(\vec E)} = \mathrm{tr}J(\vec E)= J_{00} + J_{11} + J_{22}, \\
II_{\mathbf{J}(\vec E)} = J_{00}J_{11} + J_{00}J_{22} + J_{11}J_{22} - (J_{01})^2 - (J_{02})^2 - (J_{12})^2, \\
 III_{\mathbf{J}(\vec E)} = \mathrm{det}J(\vec E) = J_{00}J_{11}J_{22} + 2J_{01}J_{02}J_{12}
  - J_{00}(J_{12})^2 \\ - J_{11}(J_{02})^2  - J_{22}(J_{01})^2.
  \end{array}
  \right.
\end{displaymath}

    Some new classes of exact solutions  of the static Maxwell system in special planarly layered media
 described by a variable coefficient  $\phi(x_2) = {x_2}^{-1}$   $(x_2 > 0)$:
\begin{equation}
  \left\{
\begin{array}{l}
  \mathrm{div} \, ({x_2}^{-1} \vec E ) = 0, \\
  \mathrm{curl}{\ \vec E} = 0,
 \end{array}
\right.
\label{Leutwiler-Maxwell-system-3}
\end{equation}
 where $ \ (E_0, E_1, E_2) := (u_0, -u_1, -u_2)$, in fact, have been studied by Leutwiler in the context of modified quaternionic analysis in $\mathbb R^3$
by means of the reduced quaternionic power series  with complex
coefficients (see, e.g., \cite{Leut:CV20,Leut:2000}).

General class of $C^1$-solutions of the system  $(\ref{Leutwiler-Maxwell-system-3})$
  is equivalently represented as general class of $C^1$-solutions  of the system
$$
(H)  \left\{
\begin{array}{l}
  x_2 \left(\frac{\partial{u_0}}{\partial{x_0}}-
  \frac{\partial{u_1}}{\partial{x_1}}-
  \frac{\partial{u_2}}{\partial{x_2}}\right)  + u_2 = 0, \\
 \frac{\partial{u_0}}{\partial{x_1}}=-\frac{\partial{u_1}}{\partial{x_0}},  \ \ \ \ \
 \frac{\partial{u_0}}{\partial{x_2}}=-\frac{\partial{u_2}}{\partial{x_0}}, \\
 \frac{\partial{u_1}}{\partial{x_2}}=\frac{\partial{u_2}}{\partial{x_1}}.
 \end{array}
\right.
$$
This system is called the system $(H)$ in honor of Hodge (see, e.g., \cite{Leut:CV20,Leut:2000}).
The system $(H)$ may be considered as a hyperbolic non-Euclidean modification of the system $(R)$
with respect to the hyperbolic metric defined on the halfspace $\{x_2 > 0\}$ by formula
(see, e.g., \cite{Ahlfors:1981,Leut:CV20,Leut:2000}):
$$
 ds^2  = \frac {d{x_0}^2 + d{x_1}^2 + d{x_2}^2}{x_2^2}.
$$

Independently new classes of exact solutions  of the static Maxwell system in special cylindrically layered media
 described by a variable coefficient  $\phi( \rho) = {\rho}^{-1}$ $( \rho > 0)$:
\begin{equation}
  \left\{
\begin{array}{l}
  \mathrm{div} \, ({\rho}^{-1} \vec E ) = 0, \\
  \mathrm{curl}{\ \vec E} = 0,
 \end{array}
\right.
\label{Bryukhov-Maxwell-system-3}
\end{equation}
 in three dimensional setting have been studied by K\"{a}hler, the author and Aksenov
by means of separation of variables in cylindrical coordinates
\cite{BrKaeh:2016,Aksenov:2017}.

 General class of $C^1$-solutions of the static Maxwell system $(\ref{Bryukhov-Maxwell-system-3})$
  is equivalently represented as general class of $C^1$-solutions of the system
$$
(A_3) \left\{
    \begin{array}{l}
      (x_1^2+x_2^2) \left(\frac{\partial{u_0}}{\partial{x_0}}-
      \frac{\partial{u_1}}{\partial{x_1}}-\frac{\partial{u_2}}{\partial{x_2}} \right) + (x_1u_1+x_2u_2)=0, \\
      \frac{\partial{u_0}}{\partial{x_1}}=-\frac{\partial{u_1}}{\partial{x_0}},
      \ \ \ \frac{\partial{u_0}}{\partial{x_2}}=-\frac{\partial{u_2}}{\partial{x_0}}, \\
      \frac{\partial{u_1}}{\partial{x_2}}=\ \ \frac{\partial{u_2}}{\partial{x_1}},
     \end{array}
  \right.
$$
 where $ \ \vec E = (u_0, -u_1, -u_2)$.
The system $(A_3)$ may be considered as an axial-hyperbolic non-Euclidean modification of the system $(R)$
with respect to the conformal metric defined outside the axis $x_0$ by formula (see, e.g., \cite{Ahlfors:1981,BrH:2011}):
$$
 ds^2  =\frac {d{x_0}^2 + d{x_1}^2 + d{x_2}^2}{\rho^2}.
$$

One of the main obstacles in applications of modified quaternionic analysis in $\mathbb R^3$
 is the problem of holistic interpretation of axially symmetric Fueter's construction  in $\mathbb R^3$
 (see, e.g., \cite{Fueter:1934,Leut:CV17,Leut:CV20}):
\begin{equation}
 F = F(x) = u_0 + i u_1 + j u_2 = u_0(x_0, \rho) + I \ u_{\rho}(x_0, \rho),
\label{Fueter}
\end{equation}
where
$$
 x = x_0 + I \rho, \ \ \ \  I = \frac{i x_1+ j x_2}{\rho}= i \cos{\theta} + j \sin{\theta}, \ \ \ \  I^2=-1,
$$
\begin{equation}
 u_1  = \frac{x_1}{\rho} u_{\rho} =  u_{\rho} \cos{\theta}, \ \ \ \
 u_2  = \frac{x_2}{\rho} u_{\rho} =  u_{\rho} \sin{\theta}.
\label{Fueter-I}
\end{equation}

 Various aspects of extensions of modified quaternionic analysis including Fueter's construction as a core element (see, e.g., \cite{LeZe:CMFT2004})
and their applications were  discussed by Leutwiler, Eriksson and
the author in Prague, November 2000 (the Workshop "Clifford Analysis
and Its Applications"). In 2003  the author characterized explicitly
class of the reduced quaternion-valued functions associated with
classical holomorphic within Fueter's construction in $\mathbb R^3$
$(\ref{Fueter})$ as joint class of analytic solutions of the system
$(H)$ and the system $(A_3)$ under the special condition (see, e.g.,
\cite{Leut:CV17,Br:2003,BrH:2011,BrKaeh:2016}):
\begin{equation}
 {u_1}{x_2}={u_2}{x_1}.
\label{spec.cond-3}
\end{equation}

\subsection{Notations}

The real algebra of quaternions $\mathbb H$ is a four dimensional skew algebra over the real field generated by real unity $1$.
Three imaginary unities $i, j,$ and $k$ satisfy to the following multiplication rules
$$
i^2 = j^2 = k^2  = ijk = -1, \quad ij = -ji = k.
$$
The independent quaternionic variable is defined as $$x = x_0 + ix_1  + jx_2  + kx_3.$$

Suppose that $ \rho = \sqrt {x_1^2+x_2^2+x_3^2}$ and  $ \rho > 0$.  We get  $x= x_0 + I \rho$,
 where  $ I = \frac{i x_1+ j x_2+ k x_3 }{\rho}$ and $ I^2=-1.$

The quaternion conjugation of $x$ is defined by the following automorphism:
$$ x \mapsto \overline{x} := x_0 - ix_1 - jx_2 - kx_3.$$

In such way, we deal with the Euclidean norm  in $\mathbb R^4$
$$\| x \|^2 :=  x \overline{x} = x_0^2 + x_1^2 + x_2^2 + x_3^2 := r^2,$$
and the identification
$$x = x_0 + ix_1  + jx_2  + kx_3 \sim (x_0,  x_1,  x_2,  x_3)$$
between $\mathbb H$ and $\mathbb R^4$ is valid. Moreover, for every non-zero value of $x$ an unique inverse value exists: $x^{-1} = \overline{x} / \| x \|^2.$

  The dependent quaternionic variable is defined as  $$u = u_0 + iu_1 + ju_2 +  ju_3 \sim (u_0, u_1, u_2, u_3).$$

The quaternion conjugation of $u$ is defined by the following automorphism:
$$ u \mapsto \overline{u} := u_0 - iu_1 - ju_2 - ku_3.$$

We have to deal with the space of reduced quaternions in case $x_3=0.$
Hereby, the  independent  reduced quaternionic variable $x= x_0 + ix_1 + jx_2$ may be identified with the vector $(x_0,x_1,x_2) \in \mathbb{R}^3$.

If $\rho > 0$, the polar angle $\varphi$ and the azimuthal angle $\theta$  are described as

$  \varphi=  \arccos \frac{x_0}{r}\ \ (0 < \varphi < \pi),
\ \ \ \theta = \arccos \frac{x_1}{\rho}\ \ (0 \leq \theta \leq 2\pi).$

 In cylindrical and spherical (sometimes called "polar") coordinates we get

 $ x = x_0 + \rho (i\cos{\theta} + j\sin{\theta})
 = r( \cos{\varphi} + i \sin{\varphi} \cos{\theta} + j \sin{\varphi} \sin{\theta} ).$

The polar angle $\varphi$ may be characterized as the argument of the reduced quaternionic variable $x$ in case $\rho > 0$:
$ \ \arg x := \varphi$ \cite{Leut:CV20}.

 \begin{defn}
Let $\Omega\subset \mathbb R^3$ be an open set.
Every continuously differentiable mapping  $u= u_0 + iu_1  + ju_2: \Omega \rightarrow \mathbb{R}^3$
is called the reduced quaternion-valued $C^1$-function $u = u(x)$ in $\Omega$.
 \end{defn}

\section{The Static Maxwell System in Special Planarly Layered Media and
$\alpha$-Hyperbolic Non-Euclidean Modification of the System $(R)$}

An original approach to building special classes of quaternion-valued solutions of the static Maxwell system
$(\ref{isotropic-electrostatic-Maxwell-system-3})$ in different layered media, where $\phi= {\phi}_0(x_0){\phi}_1(x_1){\phi}_2(x_2)$,
was developed by Kravchenko et al. in 2003 (see, e.g., \cite{KravKrav:2003,KravTach:2003}) using a quaternionic reformulation of the Dirac equation.
 A special class of quaternion-valued solutions of the
system $(\ref{isotropic-electrostatic-Maxwell-system-3})$, where ${\phi}_0(x_0) = x_0^{2p}$ $(p>0)$, ${\phi}_1(x_1)= x_1^{2m}$
$(m>0)$, ${\phi}_2(x_2) = x_2^{2n}$ $(n>0)$, was obtained by Dinh in 2021 \cite{Dinh:2021} by means of Kravchenko-generalized Dirac operators.

  General class of $C^1$-solutions of the static Maxwell system  in planarly layered media, where $\phi= {\phi}_2(x_2)>0$,
\begin{equation}
  \left\{
\begin{array}{l}
  \mathrm{div} \, ( {\phi}_2(x_2) \vec E ) = 0, \\
  \mathrm{curl}{\ \vec E} = 0
 \end{array}
\right. \label{GHR-isotropic-electrostatic-Maxwell-system-3}
\end{equation}
   is equivalently represented as general class of $C^1$-solutions of the system
 \begin{equation}
 \left\{
     \begin{array}{l}
       {\phi}_2 \left(\frac{\partial{u_0}}{\partial{x_0}}-
      \frac{\partial{u_1}}{\partial{x_1}}-
      \frac{\partial{u_2}}{\partial{x_2}}\right)- \frac{d{{\phi}_2}}{d{x_2}}u_2 = 0,  \\
      \frac{\partial{u_0}}{\partial{x_1}}=-\frac{\partial{u_1}}{\partial{x_0}}, \ \ \ \ \
      \frac{\partial{u_0}}{\partial{x_2}}=-\frac{\partial{u_2}}{\partial{x_0}}, \\
      \frac{\partial{u_1}}{\partial{x_2}}=\frac{\partial{u_2}}{\partial{x_1}},
     \end{array}
  \right.
\label{Bryukhov-hyperbolic-3}
\end{equation}
where $(u_0, u_1, u_2)=(E_0, -E_1, -E_2)$.

The continuity equation $(\ref{Lap-Bel-eq-3-general})$ is written as
 \begin{equation}
{\phi}_2 \left( \frac{{\partial}^2{h}}{{\partial{x_0}}^2} +
\frac{{\partial}^2{h}}{{\partial{x_1}}^2} +
\frac{{\partial}^2{h}}{{\partial{x_2}}^2} \right)
 +  \frac{d{{\phi}_2}}{d{x_2}} \frac{\partial{h}}{\partial{x_2}} =0.
  \label{Lap-Bel-eq-3-hyperbolic}
  \end{equation}
 Important properties of electrostatic fields may be investigated in more detail
 in case ${\phi}_2(x_2) = x_2^{-\alpha}$ $(x_2>0$, $\alpha \in \mathbb{R})$.
 We deal with the Weinstein equation in $\mathbb R^3$ (see, e.g., \cite{Weinstein:1953,Brelot-Collin:1973,AkinLeut:1994,ErOrel:2014,DinhTuyet:2020}):
\begin{equation}
 x_2 \Delta{h} - \alpha \frac{\partial{h}}{\partial{x_2}} =0.
\label{alpha-hyperbolic-3}
\end{equation}

 The static Maxwell system $(\ref{GHR-isotropic-electrostatic-Maxwell-system-3})$  is expressed as
 \begin{equation}
  \left\{
\begin{array}{l}
  \mathrm{div} \, ( x_2^{-\alpha} \vec E ) = 0, \\
  \mathrm{curl}{\ \vec E} = 0,
 \end{array}
\right.
 \label{alpha-plane-layered-electrostatic-Maxwell}
\end{equation}
and the system $(\ref{Bryukhov-hyperbolic-3})$  is simplified:
\begin{equation}
\left\{
    \begin{array}{l}
      x_2 \left(\frac{\partial{u_0}}{\partial{x_0}}-
      \frac{\partial{u_1}}{\partial{x_1}}-\frac{\partial{u_2}}{\partial{x_2}}\right)+ \alpha u_2 = 0, \\
      \frac{\partial{u_0}}{\partial{x_1}}=-\frac{\partial{u_1}}{\partial{x_0}},
      \ \ \ \frac{\partial{u_0}}{\partial{x_2}}=-\frac{\partial{u_2}}{\partial{x_0}}, \\
      \frac{\partial{u_1}}{\partial{x_2}}=\ \ \frac{\partial{u_2}}{\partial{x_1}}.
     \end{array}
  \right.
\label{eq:H_3^alpha-system}
\end{equation}
Assume that $\alpha>0$. This system may be considered as $\alpha$-hyperbolic non-Euclidean modification of the system $(R)$
 with respect to the conformal metric defined on the halfspace $\{x_2 > 0\}$ by formula:
$$
ds^2 = \frac{d{x_0}^2 + d{x_1}^2 + d{x_2}^2}{x_2^{2\alpha}}.
$$

Some new properties of exact solutions of the Weinstein equation in $\mathbb R^3$ and the system $(\ref{eq:H_3^alpha-system})$  have been studied
 in the context of hyperbolic function theory in $\mathbb R^3$ (see, e.g., \cite{ErLeut:2007,ErOrel:2009,ErOrel:2011}).

 \begin{defn}
Let $\Lambda \subset \mathbb R^3$ $ (x_2 > 0)$ be a simply connected open domain, $\alpha> 0$.
 Every exact solution of the Eq. $(\ref{alpha-hyperbolic-3})$ in $\Lambda$ is called $\alpha$-hyperbolic harmonic potential in $\Lambda$.
 \end{defn}

Nowadays solutions of the Eq. $(\ref{alpha-hyperbolic-3})$ in case
$\alpha < 0$
in the context of the theory of modified harmonic functions in
$\mathbb R^3$ (see, e.g.,
\cite{Leut:2017-CAOT,Leut:2019-AACA,Leut:2021-MMAS}) are referred to
as $-\alpha$-modified harmonic functions in $\mathbb R^3$. New
orthonormal system of polynomial modified harmonic functions on the
unit half sphere
 $S_{+}^2 = \{ (x_0, x_1, x_2): x_0^2 + x_1^2 + x_2^2 =1$, $ \ x_2 > 0 \} $ in case $\alpha = -1$ was obtained by Leutwiler in 2017
\cite{Leut:2017-CAOT} using separation of variables in spherical coordinates under condition of $\frac{\partial{h}}{\partial{\theta}} = 0$.

Meanwhile, independently specific properties of $\alpha$-hyperbolic
harmonic electrostatic potentials in three dimensional setting may
be explicitly demonstrated by means of separation of variables in
Cartesian coordinates  (see, e.g., \cite{Feshbach,Wait:1982}).

 Let us first look for a class of exact solutions of the equation $(\ref{alpha-hyperbolic-3})$
  under the first condition of separation of variables $h(x_0, x_1, x_2) =$ $g(x_0,  x_2) s(x_1)$:
$$
 s  x_2 \left( \frac{\partial{^2}{g}}{\partial{x_0}^2} +
         \frac{\partial {^2}{g}}{\partial{ x_2}^2} \right) -
\alpha s \frac{\partial{g}}{\partial{ x_2}}  +
        g x_2 \frac{d{^2}{s}}{d{x_1}^2}
 = 0.
$$

Relations
\begin{equation}
  - g \frac{d{^2}{s}}{d{x_1}^2} =
  s \left( \frac{\partial{^2}{g}}{\partial{x_0}^2} +
         \frac{\partial {^2}{g}}{\partial{ x_2}^2} \right)
 - \frac{\alpha s}{x_2} \frac{\partial{g}}{\partial{ x_2}} =
     \breve{\lambda}^2 gs  \ \ \ \ \  ( \breve{\lambda} = const \in \mathbb{R})
\label{Laplace-Beltrami-equation-3-sep}
  \end{equation}
lead to the following system of equations:
\begin{equation}
\left\{
      \begin{array}{l}
   \frac{d{^2}{s}}{d{x_1}^2} + \breve{\lambda}^2  s = 0, \\
   \frac{\partial{^2}{g}}{\partial{x_0}^2} +   \frac{\partial {^2}{g}}{\partial{x_2}^2}
 - \frac{\alpha}{x_2} \frac{\partial{g}}{\partial{x_2}}
  - \breve{\lambda}^2 g = 0.
     \end{array}
  \right.
  \label{Laplace-Beltrami equation, separation-3}
  \end{equation}

The first equation of the system $(\ref{Laplace-Beltrami equation, separation-3})$ may be solved using trigonometric functions: \\
$ s_{\breve{\lambda}} (x_1) = C_{1, \breve{\lambda}} \cos{
\breve{\lambda} x_1} + C_{2, \breve{\lambda}} \sin{ \breve{\lambda}
x_1}$, where $\breve{\lambda}\in \mathbb{Z}$; $ \ C_{1,
\breve{\lambda}}, C_{2, \breve{\lambda}}= const \in \mathbb{R}$.

Let us look for a class of exact solutions of the second equation of the system $(\ref{Laplace-Beltrami equation, separation-3})$
 under the second condition of separation of variables $g(x_0,  x_2) = \Xi(x_0)  \Upsilon(x_2)$:
$$
\Upsilon \frac{d{^2}{\Xi}}{d{x_0}^2}
+ \Xi \frac{d{^2}{ \Upsilon}}{d{x_2}^2} - \frac{\alpha \Xi}{x_2} \frac{d{ \Upsilon}}{d{x_2}} - \breve{\lambda}^2 \Xi \Upsilon = 0.
$$

 Relations
\begin{equation}
 -  \Upsilon \frac{d{^2}{\Xi}}{d{x_0}^2}  =
   \Xi \frac{d{^2}{ \Upsilon}}{d{x_2}^2}
 - \frac{\alpha \Xi}{x_2} \frac{d{ \Upsilon}}{d{x_2}}  - \breve{\lambda}^2 \Xi \Upsilon =
      - \breve{\beta}^2 \Xi \Upsilon  \ \ \ \ \  ( \breve{\beta}  = const \in \mathbb{R})
\label{equation-mu-sep-hyper}
  \end{equation}
are equivalent to the following system of ordinary differential equations:
\begin{equation}
\left\{
      \begin{array}{l}
 \frac{d{^2}{\Xi}}{d{x_0}^2} - \breve{\beta}^2  \Xi = 0, \\
   x_2^2  \frac{d{^2}{\Upsilon}}{d{x_2}^2}
 - \alpha x_2 \frac{d{\Upsilon}}{d{x_2}}
  +  ( \breve{\beta}^2 - \breve{\lambda}^2) x_2^2 \Upsilon = 0.
     \end{array}
  \right.
  \label{eq-sep-x_2-x_0-hyper}
  \end{equation}

The  first equation of the system (\ref{eq-sep-x_2-x_0-hyper}) may be solved using hyperbolic functions:  \\
$  \Xi_{\breve{\beta}}(x_0) = B_{1, \breve{\beta}} \cosh{\breve{\beta} x_0} + B_{2, \breve{\beta}} \sinh{\breve{\beta} x_0}$;
 $\ B_{1, \breve{\beta}}, B_{2, \breve{\beta}}= const \in \mathbb{R}$.

If $B_{1, \breve{\beta}} =1$ and $B_{2, \breve{\beta}}=1$, then $ \ \Xi_{\breve{\beta}}(x_0) = e^{\breve{\beta} x_0}$.

Assume that $ \breve{\lambda}^2 < \breve{\beta}^2$.
 The second equation of the system (\ref{eq-sep-x_2-x_0-hyper})
may be solved using linear independent solutions: \\
$ \Upsilon_{ \breve{\lambda}, \breve{\beta}}(x_2)= {x_2}^\frac{\alpha+1}{2} \left[ A_{1, \breve{\lambda}, \breve{\beta}} J_{\frac{\alpha+1}{2}}\left( x_2 \sqrt{\breve{\beta}^2 - \breve{\lambda}^2} \right)
+ A_{2, \breve{\lambda}, \breve{\beta}} Y_{\frac{\alpha+1}{2}}\left( x_2 \sqrt{\breve{\beta}^2 - \breve{\lambda}^2} \right) \right],$ \\
where $J_{ \breve{\nu}}(\breve{\xi})$ and $Y_{
\breve{\nu}}(\breve{\xi})$ are Bessel functions of the first and
second kind of real order $ \breve{\nu}= {\frac{\alpha + 1}{2}}$ and
real argument $\breve{\xi} = x_2 \sqrt{\breve{\beta}^2 -
\breve{\lambda}^2}$
 (see, e.g., \cite{Watson:1944,Kuzmin,Koren:2002,PolZait:Ordin-2017});
  $ A_{1, \breve{\lambda}\breve{\beta}}$, $A_{2, \breve{\lambda}, \breve{\beta}} = const \in \mathbb{R}$.

Assume that $ \breve{\lambda}^2 > \breve{\beta}^2$.
 The second equation of the system $(\ref{eq-sep-x_2-x_0-hyper})$
may be solved using linear independent solutions: \\
$ \Upsilon_{ \breve{\lambda}, \breve{\beta}}(x_2)=
{x_2}^\frac{\alpha+1}{2} \left[ A_{1, \breve{\lambda},
\breve{\beta}} J_{\frac{\alpha+1}{2}}\left(i x_2 \sqrt{
\breve{\lambda}^2 - \breve{\beta}^2} \right) + A_{2,
\breve{\lambda}, \breve{\beta}} Y_{\frac{\alpha+1}{2}}\left(i x_2
\sqrt{ \breve{\lambda}^2 - \breve{\beta}^2} \right) \right]$, where
$J_{ \breve{\nu}}(\breve{\xi})$ and $Y_{ \breve{\nu}}(\breve{\xi})$
are Bessel functions of the first and second kind of real order $
\breve{\nu}= {\frac{\alpha + 1}{2}}$
 and purely imaginary argument $\breve{\xi} = i x_2 \sqrt{ \breve{\lambda}^2 - \breve{\beta}^2}$.

This implies the following formulation.

 \begin{thm}
 A special class of exact solutions of the Weinstein equation $(\ref{alpha-hyperbolic-3})$
satisfying the relations  $(\ref{Laplace-Beltrami-equation-3-sep})$,
$(\ref{equation-mu-sep-hyper})$, $\breve{\beta}\notin \mathbb{Z}$
 in three dimensional setting may be obtained using Bessel functions of the first and second kind:
\begin{equation}
 h_{\breve{\beta}}(x_0, x_1, x_2) = \sum_{\breve{\lambda}= -\infty}^\infty \left(C_{1,\breve{\lambda}}\cos(\breve{\lambda}x_1)+C_{2,\breve{\lambda}}\sin(\breve{\lambda}x_1)\right)  g_{\breve{\lambda}, \breve{\beta}}(x_0,  x_2),
 \label{eq-Weinstein-Cart}
  \end{equation}
 \it{where}
$$
 g_{\breve{\lambda}, \breve{\beta}}(x_0, x_2) = \left(B_{1, \breve{\beta}} \cosh (\breve{\beta} x_0) + B_{2, \breve{\beta}} \sinh (\breve{\beta} x_0) \right) \Upsilon_{\breve{\lambda}, \breve{\beta}}(x_2);
$$
in case
$ \breve{\lambda}^2 < \breve{\beta}^2$
$$
 \Upsilon_{\breve{\lambda}, \breve{\beta}}(x_2) = {x_2}^\frac{\alpha+1}{2} \left[ A_{1, \breve{\lambda}, \breve{\beta}} J_{\frac{\alpha+1}{2}}\left( x_2 \sqrt{\breve{\beta}^2 - \breve{\lambda}^2} \right)
+ A_{2, \breve{\lambda}, \breve{\beta}} Y_{\frac{\alpha+1}{2}}\left( x_2 \sqrt{\breve{\beta}^2 - \breve{\lambda}^2} \right) \right]
$$
and in case
$ \breve{\lambda}^2 > \breve{\beta}^2$
$$
 \Upsilon_{\breve{\lambda}, \breve{\beta}}(x_2) = {x_2}^\frac{\alpha+1}{2} \left[ A_{1, \breve{\lambda}, \breve{\beta}} J_{\frac{\alpha+1}{2}}\left( i x_2 \sqrt{ \breve{\lambda}^2 - \breve{\beta}^2} \right)
+ A_{2, \breve{\lambda}, \breve{\beta}} Y_{\frac{\alpha+1}{2}}\left( i x_2 \sqrt{ \breve{\lambda}^2 - \breve{\beta}^2} \right) \right].
$$
 \end{thm}

Assume that $ \breve{\lambda}^2 = \breve{\beta}^2$. The second equation of the system $(\ref{eq-sep-x_2-x_0-hyper})$
leads to the Euler equation:
\begin{equation}
  x_2^2 \frac{d{^2}{ \Upsilon}}{d{ x_2}^2}
 - \alpha  x_2 \frac{d{ \Upsilon}}{d{ x_2}} = 0.
 \label{eq-Euler-Cart}
  \end{equation}
The Eq. $(\ref{eq-Euler-Cart})$ may be solved using power functions (see, e.g., \cite{PolZait:Ordin-2017}): \\
 $ \Upsilon(x_2) = A_1 x_2^{\alpha +1} + A_2$;  $\ A_1, A_2= const \in \mathbb{R}$.

 A class of electrostatic fields satisfying the relations $(\ref{Laplace-Beltrami-equation-3-sep})$,
where \\ $s(x_1) = 1$,  $h(x_0, x_1, x_2) = g(x_0, x_2)$, implies that
 the vector $\vec E$ is independent of the variable $x_1$ and $E_1= \frac{\partial{h}}{\partial{x_1}} = 0$.
The parameter $\breve{\lambda}$  vanishes,  and the second equation  of the system  $(\ref{Laplace-Beltrami equation, separation-3})$
 leads to the elliptic Euler-Poisson-Darboux equation in Cartesian coordinates  (see, e.g., \cite{Krivenkov:1957-3}):
 \begin{equation}
 x_2 \left( \frac{\partial{^2}{g}}{\partial{x_0}^2} +  \frac{\partial {^2}{g}}{\partial{x_2}^2} \right)
  - \alpha \frac{\partial{g}}{\partial{x_2}} = 0.
  \label{alpha-hyperbolic-Laplace-Beltrami-2}
  \end{equation}

 Properties of critical points of exact  solutions of the Eq. $(\ref{alpha-hyperbolic-Laplace-Beltrami-2})$ in case $\alpha = - 1$
 were investigated by Konopelchenko and Ortenzi in 2013 in the context of numerous problems of mathematical physics and catastrophe theory
  (see, e.g., \cite{Konopel,Shvartsburg,Gilmore:1993,PosStew}).

 In accordance with the Eq.
 $(\ref{alpha-hyperbolic-Laplace-Beltrami-2})$,
the system $(\ref{eq:H_3^alpha-system})$  leads to a family of Vekua
type systems investigated by Eriksson, Orelma and Sommen in the
context of hyperbolic function theory in the plane and hyperbolic
harmonic analysis \cite{ErOrel:2014-Vekua,ErOrelSommen:2016}:
\begin{equation}
\left\{
      \begin{array}{l}
   x_0  \left(\frac{\partial{u_0}}{\partial{x_0}} - \frac{\partial{u_2}}{\partial{x_2}}\right) + \alpha u_2 = 0,  \\
      \frac{\partial{u_0}}{\partial{x_2}}=
      -\frac{\partial{u_2}}{\partial{x_0}}.
     \end{array}
  \right.
\label{alpha-Vekua-type-Cart}
\end{equation}

 General class of $C^1$-solutions of Vekua
type systems $(\ref{alpha-Vekua-type-Cart})$ is equivalently
represented as general class of $C^1$-solutions of the static
Maxwell system $(\ref{alpha-plane-layered-electrostatic-Maxwell})$
in the plane $(x_0, x_2)$:
$$
\left\{
      \begin{array}{l}
        x_0  \left(\frac{\partial{E_0}}{\partial{x_0}} + \frac{\partial{E_2}}{\partial{x_2}}\right) - \alpha E_2 = 0,  \\
      \frac{\partial{E_0}}{\partial{x_2}} =
      \frac{\partial{E_2}}{\partial{x_0}},
     \end{array}
  \right.
$$
where
$$
  E_0 = \frac{\partial{g}}{\partial{x_0}}, \ \ \ \ \ \ \ \
  E_2 = \frac{\partial{g}}{\partial{x_2}}.
$$

\section{The Static Maxwell System in Special Cylindrically Layered Media and
$\alpha$-Axial-Hyperbolic Non-Euclidean Modification of the System $(R)$}

Two important classes of meridional and transverse electrostatic fields in cylindrically layered media, where   $\phi= \phi( \rho)>0$:
\begin{equation}
  \left\{
\begin{array}{l}
  \mathrm{div} \, ( \phi ( \rho) \vec E ) = 0, \\
  \mathrm{curl}{\ \vec E} = 0,
 \end{array}
\right.
\label{axial-isotropic-electrostatic-Maxwell-system-3}
\end{equation}
in cylindrical and Cartesian coordinates   were investigated by
Khmelnytskaya, Kravchenko and Oviedo in 2010 by means of applied
pseudoanalytic function theory \cite{KhmKravOv:2010,Krav:2009}. In
case of meridional fields the vector $\vec E$ is independent of the
azimuthal angle $\theta$, herewith $E_{\theta} =
\frac{\partial{h}}{\partial{\theta}} = 0$. In case of transverse
fields the vector $\vec E$ is independent of the longitudinal
variable $x_0$, herewith $E_0= \frac{\partial{h}}{\partial{x_0}} =
0$.

As seen from the system $(\ref{Bryukhov-Maxwell-system-3})$, axially
symmetric extensions of the system $(A_3)$ lead to investigation of
electrostatic fields in cylindrically layered media.

Meanwhile, general class of $C^1$-solutions of the system  $(\ref{axial-isotropic-electrostatic-Maxwell-system-3})$
   is equivalently represented as general class of $C^1$-solutions of the system
 \begin{equation}
  \left\{
     \begin{array}{l}
       \phi(\rho) \left(\frac{\partial{u_0}}{\partial{x_0}}-
      \frac{\partial{u_1}}{\partial{x_1}}-
      \frac{\partial{u_2}}{\partial{x_2}}\right)- \left( \frac{\partial{ \phi(\rho) }}{\partial{x_1}}u_1 + \frac{\partial{ \phi(\rho) }}{\partial{x_2}}u_2\right)=0,  \\
      \frac{\partial{u_0}}{\partial{x_1}}=-\frac{\partial{u_1}}{\partial{x_0}}, \ \ \ \ \
      \frac{\partial{u_0}}{\partial{x_2}}=-\frac{\partial{u_2}}{\partial{x_0}}, \\
      \frac{\partial{u_1}}{\partial{x_2}}=\frac{\partial{u_2}}{\partial{x_1}},
     \end{array}
  \right.
\label{Bryukhov-axial-3}
\end{equation}
where  $ \vec E = (u_0, -u_1, -u_2)$.

 The equation  $(\ref{Lap-Bel-eq-3-general})$ in cylindrically layered media is written as
$$
\phi \left( \frac{{\partial}^2{h}}{{\partial{x_0}}^2} + \frac{{\partial}^2{h}}{{\partial{x_1}}^2} + \frac{{\partial}^2{h}}{{\partial{x_2}}^2} \right)
 +  \frac{d{\phi}}{d{\rho}} \left( \frac{\partial{h}}{\partial{x_1}} \cos{\theta} + \frac{\partial{h}}{\partial{x_2}} \sin{\theta} \right) =0.
$$

 Suppose that $\phi ( \rho) =  \rho^{-\alpha}$ $ (\rho > 0, \ \alpha \in \mathbb{R})$.
We deal with the following axially symmetric elliptic equation in $\mathbb R^3$:
 \begin{equation}
(x_1^2+x_2^2)\Delta{h} - \alpha \left(  x_1\frac{\partial{h}}{\partial{x_1}} + x_2\frac{\partial{h}}{\partial{x_2}}\right)  =0.
  \label{axially-Kravchenko-3-alpha}
  \end{equation}

 \begin{rem}
 The invariance of solutions of the Eq. $(\ref{axially-Kravchenko-3-alpha})$ under M\"{o}bius transformations
in comparison with solutions of the Weinstein equation in $\mathbb
R^3$ $(\ref{alpha-hyperbolic-3})$ raises important issues for
consideration \cite{AkinLeut:1994}.
 \end{rem}

 The static Maxwell system $(\ref{axial-isotropic-electrostatic-Maxwell-system-3})$ is expressed as
\begin{equation}
  \left\{
\begin{array}{l}
  \mathrm{div} \, ( \rho^{-\alpha} \vec E ) = 0, \\
  \mathrm{curl}{\ \vec E} = 0,
 \end{array}
\right.
 \label{alpha-axial-isotropic-electrostatic-Maxwell-system-3}
\end{equation}
 and the system $(\ref{Bryukhov-axial-3})$  is simplified:
\begin{equation}
 \left\{
    \begin{array}{l}
      (x_1^2+x_2^2) \left(\frac{\partial{u_0}}{\partial{x_0}}-
      \frac{\partial{u_1}}{\partial{x_1}}-\frac{\partial{u_2}}{\partial{x_2}}\right) + \alpha(x_1u_1+x_2u_2) = 0, \\
      \frac{\partial{u_0}}{\partial{x_1}}=-\frac{\partial{u_1}}{\partial{x_0}},
      \ \ \ \frac{\partial{u_0}}{\partial{x_2}}=-\frac{\partial{u_2}}{\partial{x_0}}, \\
      \frac{\partial{u_1}}{\partial{x_2}}=\ \ \frac{\partial{u_2}}{\partial{x_1}}.
     \end{array}
  \right.
\label{eq:A_3^alpha-system}
\end{equation}
Assume that $\alpha>0$.
This system may be considered as $\alpha$-axial-hyperbolic non-Euclidean
 modification of the system $(R)$ with respect to the conformal metric defined outside the axis $x_0$ by formula:
$$
ds^2 = \frac{d{x_0}^2 + d{x_1}^2 + d{x_2}^2}{\rho^{2\alpha}}.
$$

 \begin{defn}
Let $\Lambda \subset \mathbb R^3$ $ (\rho > 0)$ be a simply connected
open domain, $\alpha> 0$.
 Every exact solution of the Eq. $(\ref{axially-Kravchenko-3-alpha})$ in $\Lambda$
 is called $\alpha$-axial-hyperbolic harmonic potential in $\Lambda$.
 \end{defn}

 \begin{rem}
 The system $(\ref{eq:A_3^alpha-system})$ in the context of contemporary function theories in higher dimensions
and applications in mathematical physics (see, e.g.,
\cite{GuHaSp:2008,GuHaSp:2016}) may be interpreted as a family of
axially symmetric generalizations of the Cauchy-Riemann system in
$\mathbb R^3$ for different values of the parameter $\alpha$.
 \end{rem}

 \begin{prop}
[The first criterion]
  Every $\alpha$-hyperbolic harmonic potential $h= h(x_0, x_1, x_2)$ in  $\Lambda \subset \mathbb R^3$ $(x_2 > 0)$
 represents an $\alpha$-axial-hyperbolic harmonic potential in $\Lambda$ if and only if
\begin{equation}
 x_2 \frac{\partial{h}}{\partial{x_1}} = x_1 \frac{\partial{h}}{\partial{x_2}}.
\label{meridional-condition}
\end{equation}
 \end{prop}

 As seen, necessary and sufficient condition $(\ref{meridional-condition})$ of joint class of $\alpha$-hyperbolic harmonic and $\alpha$-axial-hyperbolic harmonic potentials coincides
with the special condition $(\ref{spec.cond-3})$ of joint class of analytic solutions of the system $(H)$ and the system  $(A_3)$.

 Some specific properties of $\alpha$-axial-hyperbolic harmonic electrostatic potentials in three dimensional setting may be explicitly demonstrated
by means of separation of variables  in cylindrical coordinates (see, e.g., \cite{Feshbach,BrKaeh:2016,Aksenov:2005}).

 The Eq. $(\ref{axially-Kravchenko-3-alpha})$ in cylindrical coordinates  may be written as
 \begin{equation}
 \rho^2 \left( \frac{\partial{^2}{h}}{\partial{x_0}^2} +  \frac{\partial {^2}{h}}{\partial{\rho}^2} \right)
  - (\alpha -1) \rho \frac{\partial{h}}{\partial{\rho}}
 + \frac{\partial {^2}{h}}{\partial{\theta}^2} = 0.
  \label{alpha-axial-hyperbolic-3-cyl}
  \end{equation}

 Let us first look for a class of exact solutions of the Eq. $(\ref{alpha-axial-hyperbolic-3-cyl})$
 under the first condition of separation of variables $h(x_0, \theta, \rho) =$ $g(x_0, \rho) s( \theta)$:
$$
 s( \theta)  \rho^2 \left( \frac{\partial{^2}{g}}{\partial{x_0}^2} +
         \frac{\partial {^2}{g}}{\partial{\rho}^2} \right) -
 s( \theta)(\alpha -1) \rho \frac{\partial{g}}{\partial{\rho}}  +
        g \frac{\partial {^2}{s}}{\partial{\theta}^2}
 = 0.
$$

Relations
\begin{equation}
    -   \frac{1}{ s } \frac{\partial {^2}{s}}{\partial{\theta}^2} =
  \frac{ \rho^2 }{ g } \left( \frac{\partial{^2}{g}}{\partial{x_0}^2} +
         \frac{\partial {^2}{g}}{\partial{\rho}^2} \right)
 - \frac{(\alpha -1)\rho}{g} \frac{\partial{g}}{\partial{\rho}} =
   \breve{\lambda}^2  \ \ \ \ \  ( \breve{\lambda}  = const \in  \mathbb R )
\label{Laplace-Beltrami-equation-3-cyl-sep}
  \end{equation}
lead to the following system of equations:
\begin{equation}
\left\{
      \begin{array}{l}
   \frac{d{^2}{s}}{d{\theta}^2} + \breve{\lambda}^2  s = 0, \\
   \frac{\partial{^2}{g}}{\partial{x_0}^2} +   \frac{\partial {^2}{g}}{\partial{\rho}^2}
 - \frac{(\alpha -1)}{\rho} \frac{\partial{g}}{\partial{\rho}}
  -   \frac{ \breve{\lambda}^2 }{ \rho^2 } g = 0.
     \end{array}
  \right.
  \label{system cylindrical alpha, azimuthal separation}
  \end{equation}

The first equation of the system $(\ref{system cylindrical alpha, azimuthal separation})$  may be solved using trigonometric functions: \\
$ s_{\breve{\lambda}}(\theta) = C_{1, \breve{\lambda}} \cos{
\breve{\lambda} \theta} + C_{2, \breve{\lambda}} \sin{
\breve{\lambda} \theta}$, where $\breve{\lambda}\in \mathbb{Z}$; $ \
C_{1, \breve{\lambda}}, C_{2, \breve{\lambda}}= const \in \mathbb
R$.

Let us look for a class of exact solutions of the second equation of the system $(\ref{system cylindrical alpha, azimuthal separation})$
under the second condition of separation of variables $g(x_0,  \rho) = \Xi(x_0)  \Upsilon(\rho)$:
$$
 \frac{1}{\Xi}  \frac{d{^2}{\Xi}}{d{x_0}^2} +   \frac{1}{ \Upsilon} \frac{d{^2}{ \Upsilon}}{d{\rho}^2}
 - \frac{(\alpha -1)} { \Upsilon \rho} \frac{d{ \Upsilon}}{d{\rho}}
  - \frac{ \breve{\lambda}^2}{\rho^2} = 0.
$$

I. On the one hand, relations
\begin{equation}
  -   \frac{1}{\Xi} \frac{d{^2}{\Xi}}{d{x_0}^2} =
    \frac{1}{ \Upsilon} \frac{d{^2}{ \Upsilon}}{d{\rho}^2}
 - \frac{(\alpha -1)} { \Upsilon \rho} \frac{d{ \Upsilon}}{d{\rho}}
  - \frac{ \breve{\lambda}^2}{\rho^2} =
     - \breve{\beta}^2  \ \ \ \ \  ( \breve{\beta}  = const \in \mathbb R )
\label{equation-beta-sep-hyper-cyl}
  \end{equation}
are equivalent to the following system of ordinary differential equations:
\begin{equation}
\left\{
      \begin{array}{l}
    \frac{d{^2}{\Xi}}{d{x_0}^2} - \breve{\beta}^2  \Xi = 0, \\
 \rho^2 \frac{d{^2}{ \Upsilon}}{d{\rho}^2}
 - (\alpha -1) \rho \frac{d{ \Upsilon}}{d{\rho}}
  + ( \breve{\beta}^2 \rho^2 - \breve{\lambda}^2) \Upsilon = 0.
     \end{array}
  \right.
  \label{eq-sep-x_2-x_0-hyper-cyl}
  \end{equation}
The first equation of the system $(\ref{eq-sep-x_2-x_0-hyper-cyl})$ may be solved using hyperbolic functions: \\
$  \Xi_{\breve{\beta}}(x_0) = B_{1, \breve{\beta}} \cosh{\breve{\beta} x_0} + B_{2, \breve{\beta}} \sinh{\breve{\beta} x_0}$;
 $\ B_{1, \breve{\beta}}, B_{2, \breve{\beta}}= const \in \mathbb R$.

If $B_{1, \breve{\beta}} = 1$ and $B_{2, \breve{\beta}}= 1$, then $ \ \Xi_{\breve{\beta}}(x_0) = e^{\breve{\beta} x_0}$  (see, e.g., \cite{BrKaeh:2016}).

Assume that $\breve{\beta} \neq 0$.
The second equation of the system $(\ref{eq-sep-x_2-x_0-hyper-cyl})$
may be solved using linear independent solutions: \\
$  \Upsilon_{\breve{\lambda}, \breve{\beta}}(\rho) = {\rho}^\frac{\alpha}{2}
\left[ A_{1, \breve{\lambda}, \breve{\beta}} J_{\frac{\sqrt{\alpha^2 + 4\breve{\lambda}^2}}{2}}( \rho \breve{\beta})
+ A_{2, \breve{\lambda}, \breve{\beta}} Y_{\frac{\sqrt{\alpha^2 + 4\breve{\lambda}^2}}{2}}( \rho \breve{\beta}) \right],$ \\
where $J_{ \breve{\nu}}(\breve{\xi})$ and $Y_{
\breve{\nu}}(\breve{\xi})$ are Bessel functions of the first and
second kind of real order $ \breve{\nu}= {\frac{\sqrt{\alpha^2 +
4\breve{\lambda}^2}}{2}}$ and real argument $\breve{\xi} = \rho
\breve{\beta}$;
  $ \ A_{1, \breve{\lambda}, \breve{\beta}}$, $A_{2, \breve{\lambda}, \breve{\beta}}= const \in \mathbb R$.

This implies the following formulation.

 \begin{thm}
 A special class of exact solutions of the Eq. $(\ref{alpha-axial-hyperbolic-3-cyl})$
satisfying the relations  $(\ref{Laplace-Beltrami-equation-3-cyl-sep})$, $(\ref{equation-beta-sep-hyper-cyl})$, $\breve{\beta} \neq 0$
 in three dimensional setting  may be obtained using Bessel functions of the first and second kind in cylindrical coordinates:
$$
 h_{\breve{\beta}}(x_0, \theta, \rho) = \sum_{\breve{\lambda}= -\infty}^\infty \left(C_{1,\breve{\lambda}}\cos(\breve{\lambda}\theta)+C_{2,\breve{\lambda}}\sin(\breve{\lambda}\theta)\right)  g_{\breve{\lambda}, \breve{\beta}}(x_0, \rho),
$$
where
$$
 g_{\breve{\lambda}, \breve{\beta}}(x_0, \rho) = \left(B_{1, \breve{\beta}} \cosh (\breve{\beta} x_0) + B_{2, \breve{\beta}} \sinh (\breve{\beta} x_0) \right) \Upsilon_{\breve{\lambda}, \breve{\beta}}(\rho)
$$
and
$$
 \Upsilon_{\breve{\lambda}, \breve{\beta}}(\rho) = {\rho}^\frac{\alpha}{2}
\left[ A_{1, \breve{\lambda}, \breve{\beta}} J_{\frac{\sqrt{\alpha^2 + 4\breve{\lambda}^2}}{2}}( \rho \breve{\beta})
+ A_{2, \breve{\lambda}, \breve{\beta}} Y_{\frac{\sqrt{\alpha^2 + 4\breve{\lambda}^2}}{2}}( \rho \breve{\beta}) \right].
$$
 \end{thm}

 \begin{rem}
Suppose that a set of solutions of the Eq. $(\ref{alpha-axial-hyperbolic-3-cyl})$ satisfying the relations $(\ref{Laplace-Beltrami-equation-3-cyl-sep})$, $(\ref{equation-beta-sep-hyper-cyl})$,
where $ \Xi(x_0) = 1$, $ \breve{\beta} = 0$.
 Conditions of transverse fields are fulfilled, where $h(x_0, \theta, \rho) = \Upsilon(\rho) s(\theta)$,
 $ \ E_1 = \frac{d{ \Upsilon}}{d{\rho}} s(\theta) \cos{\theta}$  $- \Upsilon(\rho) \frac{d{s}}{d{\theta}} \frac{ \sin{\theta}}{\rho}$,
$ \ E_2 = \frac{d{ \Upsilon}}{d{\rho}} s(\theta) \sin{\theta}$ $+ \Upsilon(\rho) \frac{d{s}}{d{\theta}} \frac{ \cos{\theta}}{\rho}$.

 The Eq. $(\ref{alpha-axial-hyperbolic-3-cyl})$ in cylindrical coordinates is represented as
$$
 \rho^2 \frac{\partial {^2}{h}}{\partial{\rho}^2}
  - (\alpha -1) \rho \frac{\partial{h}}{\partial{\rho}}
 + \frac{\partial {^2}{h}}{\partial{\theta}^2} = 0,
$$
whereas the second equation  of the system $(\ref{eq-sep-x_2-x_0-hyper-cyl})$ takes the form of the Euler equation:
\begin{equation}
  \rho^2 \frac{d{^2}{ \Upsilon(\rho)}}{d{\rho}^2}
 - (\alpha -1) \rho \frac{d{ \Upsilon(\rho)}}{d{\rho}}
   - \breve{\lambda}^2 \Upsilon(\rho) = 0.
 \label{eq-Euler-cyl}
  \end{equation}
 The Eq. $(\ref{eq-Euler-cyl})$ may be solved using power functions (see, e.g., \cite{PolZait:Ordin-2017}): \\
$ \ \Upsilon(\rho) = A_{1, \breve{\lambda}} \rho^{\frac{ \alpha +
\sqrt{\alpha^2 + 4\breve{\lambda}^2}}{2}} +$ $ A_{2,
\breve{\lambda}} \rho^{\frac{ \alpha - \sqrt{\alpha^2 +
4\breve{\lambda}^2}}{2}};$ $\ A_{1, \breve{\lambda}}, A_{2,
\breve{\lambda}} = const \in \mathbb R$.

The system $(\ref{eq:A_3^alpha-system})$ leads to a family of Vekua
type systems
\begin{equation}
\left\{
      \begin{array}{l}
     (x_1^2+x_2^2) \left(\frac{\partial{u_1}}{\partial{x_1}} + \frac{\partial{u_2}}{\partial{x_2}} \right)  - \alpha(x_1u_1 + x_2u_2) = 0, \\
      \frac{\partial{u_1}}{\partial{x_2}}= \frac{\partial{u_2}}{\partial{x_1}}.
     \end{array}
  \right.
 \label{Vekua-type-x_1,x_2}
  \end{equation}

General class of $C^1$-solutions of Vekua type systems
$(\ref{Vekua-type-x_1,x_2})$ is equivalently represented as general
class of $C^1$-solutions of the static Maxwell system
$(\ref{alpha-axial-isotropic-electrostatic-Maxwell-system-3})$ in
the plane $(x_1, x_2)$:
$$
\left\{
      \begin{array}{l}
        (x_1^2+x_2^2) \left(\frac{\partial{E_1}}{\partial{x_1}} + \frac{\partial{E_2}}{\partial{x_2}} \right)  - \alpha(x_1E_1 + x_2E_2) = 0, \\
      \frac{\partial{E_1}}{\partial{x_2}} = \frac{\partial{E_2}}{\partial{x_1}}.
     \end{array}
  \right.
$$
 \end{rem}

II. On the other hand,  under the second condition of separation of variables $g(x_0,  \rho) = \Xi(x_0)  \Upsilon(\rho)$ relations
\begin{equation}
  -   \frac{1}{\Xi} \frac{d{^2}{\Xi}}{d{x_0}^2} =
    \frac{1}{ \Upsilon} \frac{d{^2}{ \Upsilon}}{d{\rho}^2}
 - \frac{(\alpha -1)} { \Upsilon \rho} \frac{d{ \Upsilon}}{d{\rho}}
  - \frac{ \breve{\lambda}^2}{\rho^2} =
     \breve{\mu}^2  \ \ \ \ \  ( \breve{\mu}  = const \in \mathbb R )
\label{equation-mu-sep-cyl}
  \end{equation}
are equivalent to the following system of ordinary differential equations:
\begin{equation}
\left\{
      \begin{array}{l}
    \frac{d{^2}{\Xi}}{d{x_0}^2} + \breve{\mu}^2  \Xi = 0, \\
 \rho^2 \frac{d{^2}{ \Upsilon}}{d{\rho}^2}
 - (\alpha -1) \rho \frac{d{ \Upsilon}}{d{\rho}}
  - ( \breve{\mu}^2 \rho^2 + \breve{\lambda}^2) \Upsilon = 0.
     \end{array}
  \right.
  \label{equation-mu-sep-x_2-x_0-cyl}
  \end{equation}
The first equation may be solved using trigonometric functions: \\
$ \ \Xi_{\breve{\mu}}(x_0) = B_{1, \breve{\mu}} \cos{\breve{\mu}
x_0} + B_{2, \breve{\mu}} \sin{\breve{\mu} x_0},$ where
$\breve{\mu}\in \mathbb{Z}$; $ \ B_{1, \breve{\mu}}, B_{2,
\breve{\mu}}= const \in \mathbb R$.

Suppose that $\breve{\mu} \neq 0$, and relations $(\ref{Laplace-Beltrami-equation-3-cyl-sep})$, $(\ref{equation-mu-sep-cyl})$ are fulfilled.
The second equation of the system $(\ref{equation-mu-sep-x_2-x_0-cyl})$
 may be solved using Bessel functions of the first kind $J_{ \breve{\nu}}(\breve{\xi})$ and second kind $Y_{ \breve{\nu}}(\breve{\xi})$
of real order $ \breve{\nu}= {\frac{\sqrt{\alpha^2 + 4\breve{\lambda}^2}}{2}}$ and purely imaginary argument $\breve{\xi} = i \rho \breve{\mu}$: \\
$ \Upsilon_{\breve{\lambda}, \breve{\mu}}(\rho)= {\rho}^\frac{\alpha}{2} \left[ A_{1, \breve{\lambda}, \breve{\mu}} J_{\frac{\sqrt{\alpha^2 + 4\breve{\lambda}^2}}{2}}\left( i \rho \breve{\mu} \right)
+ A_{2, \breve{\lambda}, \breve{\mu}} Y_{\frac{\sqrt{\alpha^2 + 4\breve{\lambda}^2}}{2}}\left( i \rho \breve{\mu} \right) \right];$
 $ A_{1, \breve{\lambda}, \breve{\mu}}, A_{2, \breve{\lambda}, \breve{\mu}} = const \in \mathbb R$.

 \begin{rem}
 New class of solutions of the Eq. $(\ref{alpha-axial-hyperbolic-3-cyl})$
satisfying the relations $(\ref{Laplace-Beltrami-equation-3-cyl-sep})$ in three dimensional setting may be obtained
using solutions of the elliptic Euler-Poisson-Darboux equation in cylindrical coordinates \cite{Aksenov:2005,Aksenov:2017}.
Change of dependent variable $g =  \rho^{\frac{ \alpha \pm \sqrt{\alpha^2 + 4\breve{\lambda}^2}}{2}} w$
 allows us to transform the second equation of the system $(\ref{system cylindrical alpha, azimuthal separation})$
into the equation
$$
\rho \left( \frac{\partial{^2}{w}}{\partial{x_0}^2} +  \frac{\partial {^2}{w}}{\partial{\rho}^2} \right)
 + \left( 1 \pm \sqrt{\alpha^2 + 4\breve{\lambda}^2} \right) \frac{\partial{w}}{\partial{\rho}} = 0.
$$
 \end{rem}

\section{The Static Maxwell System in Special Bi-Directional Planarly Layered Media and
$(\alpha_1, \alpha_2)$-Bi-Hyperbolic Non-Euclidean Modification of the System $(R)$}

Consider the specifics of exact solutions of the system
$(\ref{Bryukhov-Kaehler-3})$
 into the framework of the static Maxwell system in bi-directional planarly layered media,
where $\phi = \phi_1(x_1)\phi_2(x_2)$, $\phi_1(x_1) >0$, $\phi_2(x_2) >0$
\begin{equation}
  \left\{
\begin{array}{l}
  \mathrm{div} \, ( \phi_1(x_1)\phi_2(x_2) \vec E ) = 0, \\
  \mathrm{curl}{\ \vec E} = 0.
 \end{array}
\right.
 \label{GHBiR-isotropic-electrostatic-Maxwell-system-3}
\end{equation}

General class of $C^1$-solutions of the system $(\ref{GHBiR-isotropic-electrostatic-Maxwell-system-3})$
   is equivalently represented as general class of $C^1$-solutions of the system
 \begin{equation}
 \left\{
     \begin{array}{l}
      \phi_1(x_1)\phi_2(x_2) \left(\frac{\partial{u_0}}{\partial{x_0}}-
      \frac{\partial{u_1}}{\partial{x_1}}-  \frac{\partial{u_2}}{\partial{x_2}}\right)
-   \left( \frac{d{{\phi}_1}}{d{x_1}}u_1 + \frac{d{{\phi}_2}}{d{x_2}}u_2 \right)  = 0,  \\
      \frac{\partial{u_0}}{\partial{x_1}}=-\frac{\partial{u_1}}{\partial{x_0}}, \ \ \ \ \
      \frac{\partial{u_0}}{\partial{x_2}}=-\frac{\partial{u_2}}{\partial{x_0}}, \\
      \frac{\partial{u_1}}{\partial{x_2}}=\frac{\partial{u_2}}{\partial{x_1}},
     \end{array}
  \right.
\label{Bryukhov-Bihyperbolic-3}
\end{equation}
where  $ \vec E = (u_0, -u_1, -u_2)$.

The equation  $(\ref{Lap-Bel-eq-3-general})$ is written as
$$
 \phi_1(x_1)\phi_2(x_2) \left( \frac{{\partial}^2{h}}{{\partial{x_0}}^2} + \frac{{\partial}^2{h}}{{\partial{x_1}}^2} + \frac{{\partial}^2{h}}{{\partial{x_2}}^2} \right)
 +   \frac{d{{\phi}_1}}{d{x_1}} \frac{\partial{h}}{\partial{x_1}} +
  \frac{d{{\phi}_2}}{d{x_2}} \frac{\partial{h}}{\partial{x_2}} =0.
$$
  Suppose that  $\phi_1(x_1) =  x_1^{-\alpha_1}$, $\phi_2(x_2) =  x_2^{-\alpha_2}$ $(\alpha_1, \alpha_2 \in \mathbb{R})$.
Three dimensional elliptic equation with two singular coefficients
\begin{equation}
 \Delta{h} - \frac{\alpha_1}{x_1} \frac{\partial{h}}{\partial{x_1}} - \frac{\alpha_2}{x_2} \frac{\partial{h}}{\partial{x_2}} =0
\label{alpha-Bihyperbolic-3}
\end{equation}
is sometimes referred to as generalized bi-axially symmetric potential equation in three variables
(see, e.g., \cite{QuinnWeinacht,Zwillinger,Hasan:2007,KarNieto:2011,Dinh:2019}).

 \begin{rem}
 The invariance of solutions of the Eq. $(\ref{alpha-Bihyperbolic-3})$ under M\"{o}bius transformations
in comparison with solutions of the Weinstein equation in $\mathbb
R^3$ $(\ref{alpha-hyperbolic-3})$ raises important issues for
consideration \cite{AkinLeut:1994}.
 \end{rem}

 The static Maxwell system $(\ref{GHBiR-isotropic-electrostatic-Maxwell-system-3})$  is expressed as
 \begin{equation}
  \left\{
\begin{array}{l}
  \mathrm{div} \, ( x_1^{-\alpha_1} x_2^{-\alpha_2} \vec E ) = 0, \\
  \mathrm{curl}{\ \vec E} = 0,
 \end{array}
\right.
 \label{alpha-bi-plane-layered-electrostatic-Maxwell}
\end{equation}
and the system $(\ref{Bryukhov-Bihyperbolic-3})$  is simplified:
\begin{equation}
\left\{
    \begin{array}{l}
      (\frac{\partial{u_0}}{\partial{x_0}}-
      \frac{\partial{u_1}}{\partial{x_1}}-\frac{\partial{u_2}}{\partial{x_2}}) + \frac{\alpha_1}{x_1} u_1 + \frac{\alpha_2}{x_2} u_2 = 0, \\
      \frac{\partial{u_0}}{\partial{x_1}}=-\frac{\partial{u_1}}{\partial{x_0}},
      \ \ \ \frac{\partial{u_0}}{\partial{x_2}}=-\frac{\partial{u_2}}{\partial{x_0}}, \\
      \frac{\partial{u_1}}{\partial{x_2}}=\ \ \frac{\partial{u_2}}{\partial{x_1}}.
     \end{array}
  \right.
\label{eq:BiH_3^alpha-system}
\end{equation}
Assume that $\alpha_1>0$, $\alpha_2>0$. This system may be considered as $(\alpha_1, \alpha_2)$-bi-hyperbolic non-Euclidean modification of the system $(R)$
 with respect to the conformal metric defined on a quarter-space $\{x_1 > 0, x_2 > 0\}$ by formula:
$$
ds^2 = \frac{d{x_0}^2 + d{x_1}^2 + d{x_2}^2}{ x_1^{2\alpha_1} x_2^{2\alpha_2}}.
$$

 \begin{defn}
Let $\Lambda \subset \mathbb R^3$ $ (x_1 > 0, x_2 > 0)$ be a simply connected open domain, $\alpha_1>0$, $\alpha_2> 0$.
 Every exact solution of the Eq. $(\ref{alpha-Bihyperbolic-3})$ in $\Lambda$
 is called $(\alpha_1, \alpha_2)$-bi-hyperbolic harmonic potential in $\Lambda$.
 \end{defn}

 \begin{prop}
[Criterion of joint class of  $(\alpha_1+ \alpha_2)$-hyperbolic harmonic and  $(\alpha_1, \alpha_2)$-bi-hyperbolic harmonic potentials]
  Every $(\alpha_1+ \alpha_2)$-hyperbolic harmonic potential $h= h(x_0, x_1, x_2)$ in  $\Lambda \subset \mathbb R^3$ $(x_1>0, x_2>0)$
  represents an $(\alpha_1, \alpha_2)$-bi-hyperbolic harmonic potential in $\Lambda$ if and only if
 $x_2 \frac{\partial{h}}{\partial{x_1}} = x_1 \frac{\partial{h}}{\partial{x_2}}$.
 \end{prop}

 \begin{proof}
 Assume that  $x_1>0$ and $x_2>0$. We get $\ x_2 \frac{\partial{h}}{\partial{x_1}} = x_1 \frac{\partial{h}}{\partial{x_2}}$ if and only if
$\ \frac{1}{x_1} \frac{\partial{h}}{\partial{x_1}} = \frac{1}{x_2} \frac{\partial{h}}{\partial{x_2}}$.
This implies that
\begin{equation}
\alpha_1 \frac{1}{x_1} \frac{\partial{h}}{\partial{x_1}} + \alpha_2 \frac{1}{x_2} \frac{\partial{h}}{\partial{x_2}} =
 (\alpha_1+ \alpha_2) \frac{1}{x_1} \frac{\partial{h}}{\partial{x_1}} =
(\alpha_1+ \alpha_2) \frac{1}{x_2} \frac{\partial{h}}{\partial{x_2}}.
  \label{bi-hyperbolic}
  \end{equation}
 \end{proof}

 \begin{rem}
 Necessary and sufficient condition  $(\ref{bi-hyperbolic})$ of joint class of $(\alpha_1+ \alpha_2)$-hyperbolic harmonic and  $(\alpha_1, \alpha_2)$-bi-hyperbolic harmonic potentials
coincides with necessary and sufficient condition $(\ref{meridional-condition})$ of joint class of
$(\alpha_1+ \alpha_2)$-hyperbolic harmonic and $(\alpha_1+ \alpha_2)$-axial-hyperbolic harmonic potentials.
 \end{rem}

  Some new properties of $(\alpha_1, \alpha_2)$-bi-hyperbolic harmonic electrostatic potentials  in three dimensional setting may be demonstrated
by means of separation of variables in Cartesian coordinates.

 Let us look for a class of exact solutions of the equation $(\ref{alpha-Bihyperbolic-3})$
  under condition of $h(x_0, x_1, x_2) =$ $g(x_0,  x_2) s(x_1)$:
$$
 s  \left( \frac{\partial{^2}{g}}{\partial{x_0}^2} +
         \frac{\partial {^2}{g}}{\partial{ x_2}^2} \right) -
\frac{s \alpha_2}{x_2} \frac{\partial{g}}{\partial{ x_2}}  +
        g \frac{d{^2}{s}}{d{x_1}^2} - \frac{ \alpha_1}{x_1} g \frac{d{s}}{d{x_1}}
 = 0.
$$

 Relations
$$
  - g \frac{d{^2}{s}}{d{x_1}^2} + \frac{ \alpha_1}{x_1} g \frac{d{s}}{d{x_1}} =
  s \left( \frac{\partial{^2}{g}}{\partial{x_0}^2} +
         \frac{\partial {^2}{g}}{\partial{x_2}^2} \right)
 - \frac{s \alpha_2}{x_2} \frac{\partial{g}}{\partial{ x_2}}  =
     \breve{\lambda}^2 gs  \ \ \ \ \  ( \breve{\lambda} = const \in \mathbb R )
$$
lead to the following system of equations:
\begin{equation}
\left\{
      \begin{array}{l}
   \frac{d{^2}{s}}{d{x_1}^2} - \frac{ \alpha_1}{x_1} \frac{d{s}}{d{x_1}} + \breve{\lambda}^2  s = 0, \\
   \frac{\partial{^2}{g}}{\partial{x_0}^2} +   \frac{\partial {^2}{g}}{\partial{x_2}^2}
 - \frac{\alpha_2}{x_2} \frac{\partial{g}}{\partial{x_2}}
  - \breve{\lambda}^2 g = 0.
     \end{array}
  \right.
  \label{Laplace-Beltrami equation, bi-sep-3}
  \end{equation}

 The second equation of the system $(\ref{Laplace-Beltrami equation, bi-sep-3})$ coincides
with the second equation of the system $(\ref{Laplace-Beltrami equation, separation-3})$ in case $\alpha_2 = \alpha$.

In contrast to the system $(\ref{Laplace-Beltrami equation, separation-3})$,
the first equation of the system $(\ref{Laplace-Beltrami equation, bi-sep-3})$ takes the form of the modified Emden-Fowler equation  (see, e.g., \cite{PolZait:Ordin-2017}).
Change of independent variable $x_1 = y_1^{\frac{1}{\alpha_1 +1}}$
allows us to transform the given equation into the Emden-Fowler equation, where  $s(y_1) :=  s(x_1(y_1))$:
$$
  \frac{d{^2}{s(y_1)}}{d{y_1}^2} + \frac{\breve{\lambda}^2}{(\alpha_1 +1)^2} y_1^{- \frac{2 \alpha_1}{\alpha_1 +1}} s(y_1) = 0.
$$

\section
{Meridional Electrostatic Fields in Special Cylindrically Layered Media
and the Elliptic Euler-Poisson-Darboux Equation in Cylindrical Coordinates}

Let us compare analytic properties of $\alpha$-hyperbolic harmonic
and $\alpha$-axial-hyperbolic harmonic potentials in cylindrical
coordinates $(x_0, \theta, \rho)$.

The Weinstein equation in $\mathbb R^3$ $(\ref{alpha-hyperbolic-3})$ in cylindrical coordinates takes the following form:
 \begin{equation}
 \rho^2 \left( \frac{\partial{^2}{h}}{\partial{x_0}^2} +  \frac{\partial {^2}{h}}{\partial{\rho}^2} \right)
  - (\alpha -1) \rho \frac{\partial{h}}{\partial{\rho}} +
\frac{\partial {^2}{h}}{\partial{\theta}^2} - \alpha \cot{\theta} \frac{\partial{h}}{\partial{\theta}} =0.
  \label{alpha-hyperbolic-3-cyl}
  \end{equation}

The axially symmetric elliptic equation in $\mathbb R^3$ $(\ref{axially-Kravchenko-3-alpha})$
 in cylindrical coordinates is transformed into the equation $(\ref{alpha-axial-hyperbolic-3-cyl})$.

 \begin{prop}
[The second criterion]
  Every $\alpha$-hyperbolic harmonic potential $h= h(x_0, x_1, x_2)$ in  $\Lambda \subset \mathbb R^3$ $(x_2 > 0)$
 represents an $\alpha$-axial-hyperbolic harmonic potential in $\Lambda$ if and only if  in cylindrical coordinates
\begin{equation}
 \frac{\partial{h}}{\partial{\theta}} = 0.
\label{azimuthal-condition}
\end{equation}
 \end{prop}

The second criterion implies class of meridional electrostatic
fields in special cylindrically layered media, where $\phi=
\phi(\rho^{-\alpha}),$  $\alpha \in \mathbb{R}$.
 Thus, joint class of exact solutions of second-order elliptic
equations in cylindrical coordinates
$(\ref{alpha-hyperbolic-3-cyl})$,
$(\ref{alpha-axial-hyperbolic-3-cyl})$
 is equivalently represented as class of exact solutions of the elliptic Euler-Poisson-Darboux equation \cite{Dzhaiani,Aksenov:2005}:
  \begin{equation}
 \rho \left( \frac{\partial{^2}{g}}{\partial{x_0}^2} +  \frac{\partial {^2}{g}}{\partial{\rho}^2} \right)
  - (\alpha -1) \frac{\partial{g}}{\partial{\rho}}
  = 0.
  \label{Euler-Poisson-Darboux equation-meridional}
  \end{equation}

The Eq. $(\ref{Euler-Poisson-Darboux equation-meridional})$  is often referred to as the generalized axially symmetric potential equation (GASPE) \cite{Colton, Zwillinger}.
 Approach of generalized axially symmetric potential theory in cylindrical coordinates
 has been initiated by Weinstein  (see, e.g., \cite{Weinstein:1948-flows,Weinstein:1948-int,Weinstein:1953,Huber:1954,Erd:1956,Gilbert:1960}).
Integral representations of exact solutions of the Eq.
$(\ref{Euler-Poisson-Darboux equation-meridional})$ as generalized
axially symmetric potentials in a simply connected domain have been
obtained by Plaksa and Gryshchuk \cite{GrPlaksa:2009}.
 Linear differential relations between solutions of the Eq. $(\ref{Euler-Poisson-Darboux equation-meridional})$
 have been obtained by Aksenov
\cite{Aksenov:2005}.

 \begin{rem}
The Eq. $(\ref{Euler-Poisson-Darboux equation-meridional})$ allows
us to investigate in more detail various mathematical models
 of meridional fields, in particular, models of electrostatic fields, temperature gradient fields and potential velocity fields in special cylindrically layered media,
where $\phi= \phi(\rho^{-\alpha}),$  $\alpha \in \mathbb{R}$.
 \end{rem}

Let us consider two special subclasses of generalized axially
symmetric potentials under condition of separation of variables
$g(x_0,  \rho) = \Xi(x_0)  \Upsilon(\rho)$.

The first special subclass is provided by hyperbolic functions:   \\
$  \Xi_{\breve{\beta}}(x_0) = B_{1, \breve{\beta}}
\cosh(\breve{\beta} x_0) + B_{2, \breve{\beta}} \sinh(\breve{\beta}
x_0)$;
$\ \breve{\beta}\in \mathbb{R}$, $\ B_{1, \breve{\beta}}, B_{2, \breve{\beta}}= const \in \mathbb R$ \\
and Bessel functions of the first and second kind of order $\frac{\alpha}{2}$ and real argument:  \\
$  \Upsilon_{0, \breve{\beta}}(\rho) = {\rho}^\frac{\alpha}{2}
\left[ A_{1, 0, \breve{\beta}} J_{\frac{\alpha}{2}}( \breve{\beta} \rho)
+ A_{2, 0, \breve{\beta}} Y_{\frac{\alpha}{2}}( \breve{\beta} \rho) \right]$;
 $ \ A_{1, 0, \breve{\beta}}$, $A_{2, 0, \breve{\beta}}= const \in \mathbb R$.

The second special subclass is provided by trigonometric functions:  \\
$ \ \Xi_{\breve{\mu}}(x_0) = B_{1, \breve{\mu}} \cos(\breve{\mu}
x_0) + B_{2, \breve{\mu}} \sin(\breve{\mu} x_0)$;
 $\ \breve{\mu}\in \mathbb{Z}$, $\ B_{1, \breve{\mu}}, B_{2, \breve{\mu}}= const \in \mathbb R$ \\
and Bessel functions of the first and second kind of order $\frac{\alpha}{2}$ and purely imaginary argument: \\
$ \Upsilon_{0, \breve{\mu}}(\rho)= {\rho}^\frac{\alpha}{2} \left[ A_{1, 0, \breve{\mu}} J_{\frac{\alpha}{2}} (i \breve{\mu} \rho)
+ A_{2, 0, \breve{\mu}} Y_{\frac{\alpha}{2}} ( i \breve{\mu} \rho) \right]$;
 $ \ A_{1, 0, \breve{\mu}}$, $A_{2, 0, \breve{\mu}}= const \in \mathbb R$.

 Every generalized axially symmetric potential $g = g(x_0, \rho)$
indicates the existence of the so-called Stokes stream function
$\hat{g} = \hat{g}(x_0, \rho)$ which is defined by the generalized
Stokes-Beltrami system in the meridian half-plane $(x_0,\rho)$
$(\rho
> 0)$ in the context of GASPT
 (see, e.g., \cite{Weinstein:1948-int,Weinstein:1953,Polozhii:1973,DinhTuyet:2020}):
\begin{equation}
\left\{
      \begin{array}{l}
       \rho^{-\alpha +1} \frac{\partial{g}}{\partial{x_0}} = \frac{\partial{\hat{g}}}{\partial{\rho}},  \\
    \rho^{-\alpha +1} \frac{\partial{g}}{\partial{\rho}}= - \frac{\partial{\hat{g}}}{\partial{x_0}}.
     \end{array}
  \right.
 \label{generalized Stokes-Beltrami}
\end{equation}
 The Stokes stream function $\hat{g} = \hat{g}(x_0, \rho)$, in contrast to generalized axially symmetric potential $g = g(x_0, \rho)$,
 satisfies the elliptic Euler-Poisson-Darboux equation
$$
  \rho \left( \frac{\partial{^2}{\hat{g}}}{\partial{x_0}^2} +  \frac{\partial {^2}{\hat{g}}}{\partial{\rho}^2} \right)
  + (\alpha -1) \frac{\partial{\hat{g}}}{\partial{\rho}} = 0.
$$

 On the other hand, the Eq. $(\ref{Euler-Poisson-Darboux equation-meridional})$ leads to a
family of Vekua type systems studied   by Sommen, Pe\~{n}a Pe\~{n}a,
Sabadini \cite{Sommen:1988,PenaSommen:2012,PenaSabSommen:2017} and
Eriksson, Orelma, Vieira \cite{ErOrelVie:2017} in the context of
monogenic functions of axial type with different values of the
parameter $\alpha$:
\begin{equation}
\left\{
      \begin{array}{l}
       \rho \left( \frac{\partial{u_0}}{\partial{x_0}} - \frac{\partial{u_{\rho}}}{\partial{\rho}} \right)  +  (\alpha -1) u_{\rho} = 0,  \\
      \frac{\partial{u_0}}{\partial{\rho}}=-\frac{\partial{u_{\rho}}}{\partial{x_0}}.
     \end{array}
  \right.
\label{A_3^alpha system-meridional}
\end{equation}
We should take into account that in our setting $ \  u_0 =
\frac{\partial{g}}{\partial{x_0}}, \ \ \ \ u_{\rho} = -
\frac{\partial{g}}{\partial{\rho}}.$

  The static Maxwell system $(\ref{alpha-axial-isotropic-electrostatic-Maxwell-system-3})$ is reduced to the following two-dimensional
system:
\begin{equation}
\left\{
      \begin{array}{l}
          \rho \left( \frac{\partial{E_0}}{\partial{x_0}} + \frac{\partial{E_{\rho}}}{\partial{\rho}} \right)  -  (\alpha -1) E_{\rho} = 0,  \\
      \frac{\partial{E_0}}{\partial{\rho}} = \frac{\partial{E_{\rho}}}{\partial{x_0}},
     \end{array}
  \right.
\label{Bryukhov-elec-meridional}
\end{equation}
where
\begin{equation}
 E_0= u_0, \ \ \ \
 E_1  = \frac{x_1}{\rho} E_{\rho} = -u_1,  \ \ \ \
E_2  = \frac{x_2}{\rho} E_{\rho} = -u_2,  \ \ \ \
 E_{\rho} = -u_{\rho}.
\label{meridional-elec-fields}
\end{equation}

The principal invariants of the $EFG$ tensor within meridional
fields in special cylindrically layered media may be demonstrated
explicitly. The $EFG$ tensor $(\ref{Hessian matrix})$ is
substantially simplified:
\begin{equation}
\left(
\begin{array}{lll}
 \left[ -\frac{\partial{E_{\rho}}}{\partial{\rho}} +\frac{E_{\rho}}{\rho} (\alpha -1) \right] &  \frac{\partial{E_{\rho}}}{\partial{x_0}} \frac{x_1}{\rho} &
 \frac{\partial{E_{\rho}}}{\partial{x_0}} \frac{x_2}{\rho} \\
\frac{\partial{E_{\rho}}}{\partial{x_0}} \frac{x_1}{\rho}  & \left( \frac{\partial{E_{\rho}}}{\partial{\rho}} \frac{x_1^2}{\rho^2}  + \frac{E_{\rho}}{\rho} \frac{x_2^2}{\rho^2}\right)  &
 \left( \frac{\partial{E_{\rho}}}{\partial{\rho}}- \frac{E_{\rho}}{\rho}\right)  \frac{x_1 x_2}{\rho^2}  \\
\frac{\partial{E_{\rho}}}{\partial{x_0}} \frac{x_2}{\rho}  & \left( \frac{\partial{E_{\rho}}}{\partial{\rho}}- \frac{E_{\rho}}{\rho}\right)  \frac{x_1 x_2}{\rho^2}  &
\left( \frac{\partial{E_{\rho}}}{\partial{\rho}} \frac{x_2^2}{\rho^2} + \frac{E_{\rho}}{\rho} \frac{x_1^2}{\rho^2}\right)
 \end{array}
\right)
\label{EFG tensor-merid}
\end{equation}

This implies the following formulation.

 \begin{thm}
 Roots of the characteristic equation $(\ref{characteristic lambda})$
 of the $EFG$ tensor  $(\ref{EFG tensor-merid})$ are given by exact formulas \\
$ \lambda_{0} =  \frac{E_{\rho}}{\rho}$; \\
$ \lambda_{1, 2} = \frac{(\alpha -1)}{2}  \frac{ E_{\rho}}{ \rho}
\pm  \sqrt{ \frac{(\alpha -1)^2}{4} \left( \frac{E_{\rho}}{ \rho}
\right)^2 - (\alpha -1) \frac{ E_{\rho}}{\rho}
\frac{\partial{E_{\rho}}}{\partial{\rho}} + \left(
\frac{\partial{E_{\rho}}}{\partial{x_0}}\right)^2 + \left(
\frac{\partial{E_{\rho}}}{\partial{\rho}} \right)^2}, $ \\
such that $ \quad \lambda_{1}\lambda_{2} = (\alpha -1) \frac{
E_{\rho}}{\rho} \frac{\partial{E_{\rho}}}{\partial{\rho}} - \left(
\frac{\partial{E_{\rho}}}{\partial{x_0}}\right)^2 - \left(
\frac{\partial{E_{\rho}}}{\partial{\rho}} \right)^2$.
 \end{thm}

 \begin{proof}
 The principal invariants of the $EFG$ tensor $(\ref{EFG tensor-merid})$  are written as

$ I_{\mathbf{J}(\vec E)} = \mathrm{div} \, \vec E = \alpha
\frac{E_{\rho}}{\rho}$,

$ II_{\mathbf{J}(\vec E)} = - \left[ \left(
\frac{\partial{E_\rho}}{\partial{x_0}}  \right)^2 + \left(
\frac{\partial{E_{\rho}}}{\partial{\rho}} \right)^2 \right] +
(\alpha -1) \frac{E_{\rho}}{\rho} \left(
\frac{\partial{E_{\rho}}}{\partial{\rho}} + \frac{E_{\rho}}{\rho}
\right),$

$ III_{\mathbf{J}(\vec E)}
 =  - \frac{E_{\rho}}{\rho} \left[  \left( \frac{\partial{E_\rho}}{\partial{x_0}}  \right)^2 +  \left( \frac{\partial{E_{\rho}}}{\partial{\rho}} \right)^2 \right]
+ (\alpha -1) \left( \frac{E_{\rho}}{ \rho} \right)^2 \frac{\partial{E_{\rho}}}{\partial{\rho}}$.

The characteristic equation $(\ref{characteristic lambda})$ into the
framework of the system $(\ref{Bryukhov-elec-meridional})$ may be
factored:
$$
  \left( \lambda - \frac{E_{\rho}}{\rho} \right)
\left[ \lambda^2 - (\alpha -1) \frac{ E_{\rho}}{\rho}\lambda + (\alpha -1) \frac{ E_{\rho}}{\rho} \frac{\partial{E_{\rho}}}{\partial{\rho}}
- \left( \frac{\partial{E_{\rho}}}{\partial{x_0}}\right)^2 - \left( \frac{\partial{E_{\rho}}}{\partial{\rho}} \right)^2 \right] = 0.
$$
 \end{proof}

 \begin{cor}
 Assume that the electric field strength $\vec E = (E_0, \frac{x_1}{\rho} E_{\rho}, \frac{x_2}{\rho} E_{\rho})$
  satisfies the system $(\ref{Bryukhov-elec-meridional})$.
 The set of degenerate points of the $EFG$ tensor $(\ref{EFG tensor-merid})$ is provided by two independent equations:
$$
{E_{\rho}}=0,\quad
\left(\frac{\partial{E_{\rho}}}{\partial{x_0}}\right)^2
+\left(\frac{\partial{E_{\rho}}}{\partial{\rho}}\right)^2
-(\alpha-1)\frac{E_{\rho}}{\rho}\frac{\partial{E_{\rho}}}{\partial{\rho}}=0.
 $$
 \end{cor}

The system $(\ref{Bryukhov-elec-meridional})$
 allows us to demonstrate substantially new properties of meridional fields in inhomogeneous and homogeneous media.

  \begin{cor} [On the zero divergence condition]
Assume that the electric field strength $\vec E = (E_0,
\frac{x_1}{\rho} E_{\rho}, \frac{x_2}{\rho} E_{\rho})$ satisfies the
system~\eqref{Bryukhov-elec-meridional}, where $\alpha \neq 0$.
Every point $x = (x_0, x_1, x_2)$, where $\mathrm{div} \, \vec E
(x_0, x_1, x_2) = 0$, is a degenerate point of the $EFG$
tensor~\eqref{EFG tensor-merid}.
\end{cor}

 Geometric properties of the $EFG$ tensor $(\ref{EFG tensor-merid})$
allow us to introduce the concept of $\alpha$-meridional mappings of
the first and second kind.

 \begin{defn}
Let $\alpha$ be a real parameter, while $\Lambda \subset \mathbb
R^3$  be a simply connected open domain, where $x_1 \neq 0, x_2 \neq
0$. Assume that an exact solution $(u_0, u_1, u_2)$ of the system
$(\ref{eq:A_3^alpha-system})$, where $\alpha \neq 0$, satisfies the
following condition: $x_2 u_1 = x_1 u_2$ in $\Lambda$. Mapping $u =
u_0 + iu_1 + ju_2: \Lambda \rightarrow \mathbb{R}^3$ is called
$\alpha$-meridional mapping of the first kind, and mapping $
\overline{u} = u_0 - iu_1 - ju_2: \Lambda \rightarrow \mathbb{R}^3$
is called $\alpha$-meridional mapping of the second kind,
respectively.
 \end{defn}

 The principal invariants of $\alpha$-meridional mappings of
the second kind coincide with the principal invariants of the $EFG$
tensor $(\ref{EFG tensor-merid})$.

\section
{{The Radially Holomorphic Potential in Electrostatics and
Meridional Models Provided by the Reduced Quaternionic
Laplace-Fueter and Fourier-Fueter Transforms of Real-Valued Original
Functions}}

 A Vekua type system $(\ref{A_3^alpha system-meridional})$ in case $\alpha =1$
 may  be considered as a Cauchy-Riemann type system in the meridian half-plane $(\rho > 0)$
(see, e.g., \cite{Weinstein:1948-flows, Weinstein:1948-int,Weinstein:1953,BrKaeh:2016}):
\begin{equation}
 \left\{
      \begin{array}{l}
        \frac{\partial{u_0}}{\partial{x_0}} - \frac{\partial{u_{\rho}}}{\partial{\rho}}  = 0,  \\
      \frac{\partial{u_0}}{\partial{\rho}}=-\frac{\partial{u_{\rho}}}{\partial{x_0}}.
     \end{array}
  \right.
\label{CR-merid}
\end{equation}

The generalized Stokes-Beltrami system  $(\ref{generalized
Stokes-Beltrami})$ in case $\alpha =1$ leads to the Cauchy-Riemann
type system in the meridian half-plane concerning functions $g =
g(x_0, \rho)$, $\hat{g} = \hat{g}(x_0, \rho)$:
\begin{equation}
\left\{
      \begin{array}{l}
        \frac{\partial{g}}{\partial{x_0}} - \frac{\partial{\hat{g}}}{\partial{\rho}}  = 0,  \\
      \frac{\partial{g}}{\partial{\rho}}=-\frac{\partial{\hat{g}}}{\partial{x_0}}.
     \end{array}
  \right.
\label{Stokes-Beltrami-1}
\end{equation}

Generalized axially symmetric potential $g = g(x_0, \rho)$ and the
Stokes stream function $\hat{g} = \hat{g}(x_0, \rho)$ in case
$\alpha = 1$ satisfy equations
  $$
   \frac{\partial{^2}{g}}{\partial{x_0}^2} +  \frac{\partial {^2}{g}}{\partial{\rho}^2}  =
   0, \quad \quad
 \frac{\partial{^2}{\hat{g}}}{\partial{x_0}^2} +  \frac{\partial {^2}{\hat{g}}}{\partial{\rho}^2}  = 0.
$$

The first-order systems  $(\ref{CR-merid})$,
$(\ref{Stokes-Beltrami-1})$ arise independently in pure mathematics
in a number of seemingly disconnected settings (see, e.g.,
\cite{Weinstein:1953,Aksenov:2005,ColSabStruppa:2009,ColSabStruppa:2016}).

 In particular, an original approach to building special classes of the reduced quaternion-valued regular functions was developed by
Gentili and Struppa in 2006 in the context of the theory of analytic
intrinsic functions on quaternions (see, e.g.,
\cite{Cullen:1965,GentStruppa:2006}). As noted by Gentili and
Struppa \cite{GentStruppa:2006}, "Cullen regular functions are
closely related to a class of functions of the reduced quaternionic
variable $x_0 + ix_1  + jx_2$, studied by
H.~Leutwiler~\cite{Leut:CV20}. This class consists of all the
solutions of a generalized Cauchy-Riemann system of equations, it
contains the natural polynomials, and supports the series expansion
of its elements as well." Nowadays Cullen regular functions are
referred to as  slice regular functions (sometimes to as "slice
monogenic" or "slice hyperholomorphic") (see, e.g.,
\cite{ColSabStruppa:2009,ColSabStruppa:2016,BrKaeh:2016}). As noted
by  K\"{a}hler and the author in 2017 \cite{BrKaeh:2016}, "These are
defined as reduced quaternion-valued functions $F$ which fulfill the
following equation $ DF =\left( \frac{\partial{}}{\partial{x_0}}+I
\frac{\partial{}}{\partial{\rho}} \right) F =0 $ on every slice
domain belonging to the plane spanned by $1$ and $I\in S^2$. Now, if
we additionally impose $F$ being of the form $F=u_0(x_0,\rho)+I
u_\rho(x_0,\rho)$, then the above definition can be written as the
Cauchy-Riemann system $(\ref {CR-merid})$."

On the other hand, a survey of the construction of Clifford regular
elementary functions was given and important properties of a class
of radially regular elementary functions including a
paravector-valued logarithm were deduced by Spr\"{o}{\ss}ig in 1999
\cite{Spr:1999}. Later the concept of radially holomorphic functions
has been developed by G\"{u}rlebeck, Habetha and Spr\"{o}{\ss}ig
\cite{GuHaSp:2008} in the context of the theory of holomorphic
functions in n-dimensional space.


 \begin{defn}
 Radial differential operator is defined as
$$
 \partial_{rad}{G}  := \frac{1}{2} \left( \frac{\partial{}}{\partial{x_0}}- I \frac{\partial{}}{\partial{\rho}} \right) G := G'  \ \ \ \ \ (G = g + I \hat{g}).
$$

Every reduced quaternion-valued function $G = g + I \hat{g}$
satisfying a Cauchy-Riemann type differential equation  in $\Lambda$  $(\rho > 0)$
\begin{equation}
\overline{\partial}_{rad}{G} :=  \frac{1}{2} \left( \frac{\partial{}}{\partial{x_0}} + I \frac{\partial{}}{\partial{\rho}} \right) G = 0
\label{CRDO-con}
\end{equation}
is called a radially holomorphic in $\Lambda$.
The reduced quaternion-conjugate function $\overline{G} = g - I \hat{g}$
is called a radially anti-holomorphic in $\Lambda$.
 \end{defn}

   The notation $\partial_{rad}{G} := G'$ has been justified in \cite{GuHaSp:2008} by some clear
   statements.
In particular, elementary functions of the reduced quaternionic variable
as elementary radially holomorphic functions in $\mathbb R^3$ satisfy the following relations:

$ \ \  \ \  \ \  \ \  \ \  \ \  \ \  \ \  \ \  \ \  [x^{n} := r^{n}(\cos{n\varphi} + I \sin{n\varphi})]' = n x^{n-1}; $

 $  \ \  \ \  \ \  \ \  \ \  \ \  \ \  \ \  \ \  \ \  [e^{x} := e^{x_0}(\cos{\rho}+I \sin{\rho})]'= e^{x};  $

$  \ \  \ \  \ \  \ \  \ \  \ \  \ \  \ \  \ \  \ \  [\cos{x} := \frac{1}{2}( e^{-Ix} + e^{Ix})]'= - \sin{x};  $

$  \ \  \ \  \ \  \ \  \ \  \ \  \ \  \ \  \ \  \ \   [\sin{x} := \frac{I}{2}( e^{-Ix} - e^{Ix})]'= \cos{x};  $

 $ \ \  \ \  \ \  \ \  \ \  \ \  \ \  \ \  \ \  \ \  [ \ln{x} := \ln r +I \varphi]' = x^{-1}.$

 The Eq. $(\ref{CRDO-con})$ implies that
$$
G' = \frac{\partial{G}}{\partial{x_0}}.
$$

 Appropriate concept of radially holomorphic primitives has been initiated in \cite{GuHaSp:2008}.
 \begin{defn}
 Suppose that  a radially holomorphic  function $G = g + I \hat{g}$ in $\Lambda$ satisfies a differential equation
$$
 G' = F,
$$
where  function $F = u_0 + I u_{\rho}$ is also a radially holomorphic in $\Lambda$.
The function $G$ is called a radially holomorphic primitive of the function $F$  in $\Lambda$.
 \end{defn}

 Let us clarify now basic properties of radially holomorphic primitives $G$ in $\Lambda$  $(\rho > 0)$.
Let us consider a curve $(l)$ of a definite direction, defined by real-valued $C^1$-functions $x_0(\xi)$ and $\rho(\xi)$ in the reduced quaternion-valued parametric form
 $ x(\xi) = x_0(\xi) + I \rho(\xi), $   $ x'(\xi) = x'_0(\xi) + I {\rho}'(\xi),$ $ \xi \in \mathbb R,$  where $I = i \cos{\theta} + j \sin{\theta}$, $ \theta = const.$
 Assume that a point $x = (x_0, \rho\cos{\theta}, \rho\sin{\theta})$ belongs to the curve $(l)$ from $x^0 = x_0^0  + I \rho^0$ to $x^1= x_0^1  + I \rho^1$ in $\Lambda$, where
 $\xi$ varies from $\xi^0$ to $\xi^1$ in the closed interval $(\xi^0, \xi^1)$, whilst taking $\xi^0 < \xi^1$ for clarity,
and $ \ x^0 = x(\xi^0), \ $ $x^1 = x(\xi^1)$ (see, e.g., \cite{LavSh,Smirnov}).

 \begin{lem} [On path independent line integral]
Let  $F= u_0 + I u_{\rho}$ be a radially holomorphic function in
$\Lambda$  $(\rho > 0)$. Any reduced quaternion-valued line integral
along the curve $(l)$
$$
 \int\limits_{x^0}^{x} F(x) dx
 := \int\limits_{x^0}^{x} (u_0 dx_0 - u_{\rho} d{\rho})
 + I \int\limits_{x^0}^{x} (u_{\rho} dx_0 + u_0 d{\rho}); \ \ \ dx = dx_0 + I d{\rho}
$$
is path independent if and only if functions $u_0 = u_0(x_0, \rho)$,
$u_{\rho} = u_{\rho}(x_0, \rho)$ satisfy the Cauchy-Riemann type
system in the meridian half-plane~\eqref{CR-merid}, such that
$$
\int\limits_{x^0}^{x} F(x) dx = \int\limits_{\xi^0}^{\xi}
F[x(\hat{\varsigma})] x'( \hat{\varsigma})d\hat{\varsigma} =
\int\limits_{x^0}^{x} dx F(x)  \ \ \ ( \hat{\varsigma} \in \mathbb
R).
$$
 \end{lem}

 \begin{defn}
The reduced quaternion-valued line integral along the curve $(l)$ in $\Lambda$
$$
  \int\limits_{x^0}^{x} F(x) dx = \int\limits_{\xi^0}^{\xi} F[x(\hat{\varsigma})] x'( \hat{\varsigma})d\hat{\varsigma}  \ \ \ ( \hat{\varsigma} \in \mathbb R)
$$
is called  an indefinite integral of radially holomorphic function
$F$  in $\Lambda$.
 \end{defn}

This implies the following formulation.

 \begin{prop} [On integral form of radially holomorphic primitives]
 Every radially holomorphic function $F = u_0 + I u_{\rho}$ in $\Lambda$  $(\rho > 0)$
has a  radially holomorphic primitive $G = g + I \hat{g}$ taking the
form of an indefinite integral
$$
 G = \int\limits_{x^0}^{x} F(x) dx + G^0,
$$
 where $G^0 = g^0 + I \hat{g}^0;$ $ \ g^0 = g(x_0^0, \rho^0),$ $ \ \hat{g}^0= \hat{g}(x_0^0, \rho^0)$.
 Functions $g = g(x_0, \rho)$, $\hat{g} = \hat{g}(x_0, \rho)$ are given by formulas:
$$
 g(x_0, \rho) = \int\limits_{x^0}^{x} (u_0 dx_0 - u_{\rho} d{\rho}) + g^0, \ \ \ \ \
 \hat{g}(x_0, \rho) = \int\limits_{x^0}^{x} (u_{\rho} dx_0 + u_0 d{\rho}) + \hat{g}^0.
$$
  \end{prop}

 \begin{rem}
Real-valued line integrals along the curve $(l)$
 in multiply connected open domains $\Omega\subset \mathbb R^3$ $(\rho > 0)$
$$
  \int\limits_{(l)}^{} (u_0 dx_0 - u_{\rho} d{\rho}), \ \ \ \ \
  \int\limits_{(l)}^{} (u_{\rho} dx_0 + u_0 d{\rho})
$$
 may  provide the multi-valued scalar potential $g = g(x_0, \rho)$ and the multi-valued Stokes stream function $\hat{g} = \hat{g}(x_0, \rho)$  in the context of GASPT
using the cyclic constants in the meridian half-plane  (see, e.g., \cite{Weinstein:1948-flows, Weinstein:1948-int,LavSh}).
 \end{rem}

 Numerous mathematical problems of two-dimensional potential fields in homogeneous media
 have been investigated by means of the complex potential and conformal
mappings of the second kind.
  In accordance  with the theory of holomorphic functions of a complex variable,
  where $f = f(z) = u + iv$, $z = x + iy$,
analytic models in electrostatics are characterized by the principal
invariants of the form $\ I_{\mathbf{J}(\vec E)} = \mathrm{tr}
\mathbf{J}(\vec E) = 0,$ $\ II_{\mathbf{J}(\vec E)} =
\det\mathbf{J}(\vec E) = - |f'(z)|^2 \leq 0$ (see e.g.,
\cite{LavSh}).

Let us now look at properties of the EFG tensor $\mathbf{J}(\vec E)$
in cylindrically layered media, where $\phi( \rho) = \rho^{-1}$,
taking into account that the system~\eqref{Bryukhov-elec-meridional}
 is expressed as
\begin{equation}
\left\{
      \begin{array}{l}
        \frac{\partial{E_0}}{\partial{x_0}} + \frac{\partial{E_{\rho}}}{\partial{\rho}}  = 0,  \\
      \frac{\partial{E_0}}{\partial{\rho}} = \frac{\partial{E_{\rho}}}{\partial{x_0}}.
     \end{array}
  \right.
\label{Bryukhov-elec}
\end{equation}

The principal invariants of the EFG tensor $(\ref{EFG
tensor-merid})$  in case $\alpha =1$
\begin{equation}
\left(
\begin{array}{lll}
 -\frac{\partial{E_{\rho}}}{\partial{\rho}} &  \frac{\partial{E_{\rho}}}{\partial{x_0}} \frac{x_1}{\rho} &
 \frac{\partial{E_{\rho}}}{\partial{x_0}} \frac{x_2}{\rho} \\
\frac{\partial{E_{\rho}}}{\partial{x_0}} \frac{x_1}{\rho}  & \left(
\frac{\partial{E_{\rho}}}{\partial{\rho}} \frac{x_1^2}{\rho^2}  +
\frac{E_{\rho}}{\rho} \frac{x_2^2}{\rho^2}\right)  &
 \left( \frac{\partial{E_{\rho}}}{\partial{\rho}}- \frac{E_{\rho}}{\rho}\right)  \frac{x_1 x_2}{\rho^2}  \\
\frac{\partial{E_{\rho}}}{\partial{x_0}} \frac{x_2}{\rho}  & \left(
\frac{\partial{E_{\rho}}}{\partial{\rho}}-
\frac{E_{\rho}}{\rho}\right)  \frac{x_1 x_2}{\rho^2}  & \left(
\frac{\partial{E_{\rho}}}{\partial{\rho}} \frac{x_2^2}{\rho^2} +
\frac{E_{\rho}}{\rho} \frac{x_1^2}{\rho^2}\right)
 \end{array}
\right)
\label{EFG tensor-merid-1}
\end{equation}
 are written as
$$
 I_{\mathbf{J}(\vec E)} = \frac{E_{\rho}}{\rho},
 \ \ \ \ \ II_{\mathbf{J}(\vec E)} = - \left[  \left(
\frac{\partial{E_\rho}}{\partial{x_0}}  \right)^2 +  \left(
\frac{\partial{E_{\rho}}}{\partial{\rho}} \right)^2 \right],
$$
$$
 III_{\mathbf{J}(\vec E)}
 =  - \frac{E_{\rho}}{\rho} \left[  \left( \frac{\partial{E_\rho}}{\partial{x_0}}  \right)^2
 +  \left( \frac{\partial{E_{\rho}}}{\partial{\rho}} \right)^2 \right].
$$

The second principal invariant of the $EFG$ tensor $(\ref{EFG
tensor-merid-1})$ satisfies the inequality $II_{\mathbf{J}(\vec E)}
\leq 0$. The third principal invariant of the $EFG$ tensor
$(\ref{EFG tensor-merid-1})$ satisfies the inequality $\
III_{\mathbf{J}(\vec E)} < 0$ if and only if $E_{\rho}
> 0$.

 \begin{cor}
 Roots of the characteristic equation of the $EFG$ tensor  $(\ref{EFG tensor-merid-1})$
 in case $\alpha =1$ are  given by formulas:
\begin{equation}
 \lambda_{0}  = \frac{E_{\rho}}{\rho} = \mathrm{div} \, \vec E, \quad
 \quad \lambda_{1, 2} = \pm  \sqrt{\left(
\frac{\partial{E_{0}}}{\partial{x_0}}\right)^2 + \left(
\frac{\partial{E_{\rho}}}{\partial{x_0}} \right)^2} =  \pm |F'|.
\label{lambda-merid-1}
\end{equation}
 \end{cor}

Exact formulas~\eqref{lambda-merid-1} allow us to demonstrate
explicitly the geometric specifics of the $EFG$ tensor~\eqref{EFG
tensor-merid-1} in some cylindrically layered media, in contrast to
the geometric specifics of conformal mappings of the second kind in
homogeneous media.
 These formulas have been missed in applications of
  pseudoanalytic function theory \cite{Krav:2009,KhmKravOv:2010},
   modified quaternionic analysis in $\mathbb R^3$ (see, e.g., \cite{Leut:CV20,Leut:More95,Leut:Rud96}),
  the theory of Cullen regular (slice regular) functions (see, e.g.,
  \cite{GentStruppa:2006,ColSabStruppa:2009,ColSabStruppa:2016}) and
  the theory of holomorphic functions in $n$-dimensional space \cite{GuHaSp:2008}.

 \begin{defn}
 Radially holomorphic primitive $G = g + I \hat{g}$ in simply connected open domains $\Lambda$ $ ( \rho > 0)$
 in the context of the system $(\ref{Bryukhov-elec})$ is called the radially holomorphic potential.
 \end{defn}

\begin{ex}
The reduced quaternionic M\"{o}bius transformation with real coefficients:
  $F(x) = (ax + b)(cx + d)^{-1} = -\frac{1}{c^2}(x + \frac{d}{c})^{-1} + \frac{a}{c}$, where  $ \ ad - bc = 1$;
 $ a,b,c,d \in \mathbb R$  (see, e.g., \cite {Ahlfors:1981, Leut:CV20, BrW:2011}). \\
We deal with a radially anti-holomorphic function
$ \overline{F}(x) = -\frac{1}{c^2}(\overline{x} + \frac{d}{c})^{-1} + \frac{a}{c}$.

The radially holomorphic potential in simply connected open domains $\Lambda \subset \mathbb R^3$ $( \rho > 0)$
takes the form  $ \ G = -\frac{1}{c^2} \ln{(x + \frac{d}{c})} + \frac{a}{c}x$.

We get a meridional electrostatic field, generalizing linear superposition of plane single sink of intensity $N =-\frac{1}{c^2}\ $ and constant electrostatic field $ \frac{a}{c}$,
where $ \ E_0 = - \frac{1}{c^2} \frac{(x_0 + \frac{d}{c})}{[(x_0 + \frac{d}{c})^2 +\rho^2]} + \frac{a}{c}$,  $ \ E_{\rho} = - \frac{1}{c^2} \frac{\rho}{[(x_0 + \frac{d}{c})^2 + \rho^2]}$.

 The $EFG$ tensor $\mathbf{J}(\vec E)$ is  written as
\begin{equation}
  \frac{1}{c^2}
 \left(
  \begin{array}{lll}
 \frac{(x_0+ \frac{d}{c})^2-x_1^2-x_2^2}{[(x_0+ \frac{d}{c})^2+x_1^2 + x_2^2]^2} & \frac{2(x_0 + \frac{d}{c})x_1}{[(x_0 + \frac{d}{c})^2 + x_1^2 + x_2^2]^2} & \frac{2(x_0 + \frac{d}{c})x_2}{[(x_0 + \frac{d}{c})^2 + x_1^2 + x_2^2]^2} \\
 \frac{2(x_0 + \frac{d}{c})x_1}{[(x_0 + \frac{d}{c})^2 + x_1^2 + x_2^2]^2}  & \frac{-(x_0+ \frac{d}{c})^2+x_1^2-x_2^2}{[(x_0 + \frac{d}{c})^2 + x_1^2 + x_2^2]^2}  &    \frac{-2x_1 x_2}{[(x_0 + \frac{d}{c})^2 + x_1^2 + x_2^2]^2} \\
 \frac{2(x_0 + \frac{d}{c})x_2}{[(x_0 + \frac{d}{c})^2 + x_1^2 + x_2^2]^2}  & \frac{-2x_1 x_2}{[(x_0 + \frac{d}{c})^2 + x_1^2 + x_2^2]^2}  &   \frac{-(x_0+ \frac{d}{c})^2-x_1^2+x_2^2}{[(x_0 + \frac{d}{c})^2 + x_1^2 + x_2^2]^2}
 \end{array}
 \right)
\label{Moebius}
\end{equation}

Roots of the characteristic equation  $(\ref{characteristic lambda})$  of the $EFG$ tensor  $(\ref{Moebius})$ are given by formulas:
 $$
 \lambda_{0} = \frac{-1}{c^2 [(x_0 + \frac{d}{c})^2 +\rho^2]}, \ \ \ \ \
\lambda_{1, 2} = \pm \frac{ \sqrt{ (x_0 + \frac{d}{c})^4 + \rho^4} }{c^2 [(x_0 + \frac{d}{c})^2 + \rho^2]}.
 $$
Thus, the set of degenerate points of the $EFG$ tensor
$(\ref{Moebius})$ is empty.
 \end{ex}

\begin{ex}
The reduced quaternionic cubic polynomial with real coefficients:
$\ F(x) = a_3 x^3 + a_1 x$;   $ \ a_3, a_1 \in \mathbb R$. \\
We deal with a radially anti-holomorphic function $ \ \overline{F}(x) = a_3 \overline{x}^3 + a_1 \overline{x}. $

The radially holomorphic potential in simply connected open domains $\Lambda\subset \mathbb R^3$  $(\rho > 0)$
takes the form $G = \frac{a_3}{4} x^4 + \frac{a_1}{2} x^2$.

We get a meridional electrostatic field, where  $ \ \vec E = \vec
E^1 + \vec E^2$,
 $$
\vec E^1 =  \left( E_0^1, \frac{x_1}{\rho} E_{\rho}^1, \frac{x_2}{\rho} E_{\rho}^1 \right),  \ \ \ \ \  \vec E^2 =  \left( E_0^2, \frac{x_1}{\rho} E_{\rho}^2, \frac{x_2}{\rho} E_{\rho}^2 \right),
$$
 $$
E_0^1 = a_3 (x_0^2 - 3 x_1^2 - 3 x_2^2) x_0,  \ \ \ \ \  E_{\rho}^1 = a_3 (-3 x_0^2 + x_1^2 +  x_2^2) \rho,
 $$
 $$
E_0^2 = a_1 x_0,  \ \ \ \ \  E_{\rho}^2 = - a_1 \rho .
 $$

 The $EFG$ tensors $\mathbf{J}(\vec E^1)$ and $\mathbf{J}(\vec E^2)$ are  written as
 $$
 a_3
\left(
\begin{array}{lll}
  (3x_0^2 - 3 x_1^2 - 3 x_2^2)  &   -6x_0 x_1    &  -6x_0 x_2 \\
  -6 x_0 x_1 &  (-3 x_0^2 + 3x_1^2 + x_2^2)   & 2x_1 x_2         \\
  -6 x_0 x_2  &  2x_1 x_2      &   (-3 x_0^2 +  x_1^2 + 3x_2^2)
 \end{array}
\right),
 $$
 $$
  a_1
 \left(
  \begin{array}{lll}
  1  &   0 &   0 \\
  0  &  - 1   & 0 \\
  0  & 0  &  - 1
 \end{array}
 \right).
 $$

The zero divergence condition leads to the well-known quadratic algebraic equation:
 \begin{equation}
\mathrm{div} \, \vec E = \frac{E_{\rho}}{\rho} = -3 a_3 x_0^2 + a_3
(x_1^2 + x_2^2) - a_1 = 0.
  \label{eq: non-degenerate quadric surfaces}
 \end{equation}

If $ \ a_1 \ne 0$, the Eq. (\ref{eq: non-degenerate quadric
surfaces}) provides two types of non-degenerate quadric surfaces  of
revolution in $\mathbb R^3=\{(x_0,x_1,x_2)\}$ (with the axis of
revolution $x_0$). If the signs of the coefficients $a_1$ and $a_3$
are the same,  we deal with a one-sheeted circular hyperboloid as a
surface of negative Gaussian curvature. If the signs of the
coefficients $a_1$ and $a_3$ are opposite, we deal with a
two-sheeted circular hyperboloid
 as a surface of positive Gaussian curvature (see, e.g., \cite{HilCV:1999}).

If $ \ a_1 = 0$,
we deal with a circular cone as a surface of zero Gaussian curvature:
$$
3 x_0^2 - (x_1^2 + x_2^2)  = 0.
$$

 Roots of the characteristic equation  $(\ref{characteristic lambda})$  of the $EFG$ tensor $\mathbf{J}(\vec E) = \mathbf{J}(\vec E^1) + \mathbf{J}(\vec E^2)$ are given by formulas:
 $$
 \lambda_{0} = -3 a_3 x_0^2 + a_3 \rho^2 - a_1, \ \ \ \ \
\lambda_{1, 2} = \pm  \sqrt{(3 a_3 x_0^2 - 3 a_3 \rho^2 + a_1)^2 + 36 a_3^2 x_0^2 \rho^2}.
 $$
 \end{ex}

 \begin{ex}
Linear superposition of the reduced quaternionic power functions with negative exponents:
$ \ F(x)= a_{-1} x^{-1}+ a_{-2} x^{-2}$;  $ \ a_{-1}, a_{-2} \in \mathbb R$. \\
We deal with a radially anti-holomorphic function $ \ \overline{F}(x) = a_{-1} \overline{x^{-1}} + a_{-2} \overline{x^{-2}}. $

The radially holomorphic potential in simply connected open domains $\Lambda\subset \mathbb R^3$  $(\rho > 0)$
takes the form $G = a_{-1} \ln{x} - a_{-2} x^{-1}$.

We get a meridional electrostatic field, where  $ \ \vec E = \vec
E^1 + \vec E^2$,
 $$
\vec E^1 =  \left( E_0^1, \frac{x_1}{\rho} E_{\rho}^1, \frac{x_2}{\rho} E_{\rho}^1 \right),  \ \ \ \ \  \vec E^2 =  \left( E_0^2, \frac{x_1}{\rho} E_{\rho}^2, \frac{x_2}{\rho} E_{\rho}^2 \right),
$$
 $$
E_0^1 = \frac { a_{-1}x_0}{x_0^2 + \rho^2},  \ \ \ \ \  E_{\rho}^1 = \frac{ a_{-1} \rho}{x_0^2 +\rho^2},
 $$
 $$
E_0^2 =  \frac{a_{-2}(x_0^2 - \rho^2)}{(x_0^2 + \rho^2)^2},  \ \ \ \ \  E_{\rho}^2 = \frac{2a_{-2} x_0 \rho}{(x_0^2 + \rho^2)^2}.
 $$

 The $EFG$ tensors $\mathbf{J}(\vec E^1)$ and $\mathbf{J}(\vec E^2)$ are  written as
$$
a_{-1}
 \left(
  \begin{array}{lll}
 \frac{-x_0^2 + x_1^2 + x_2^2}{ (x_0^2 + x_1^2 + x_2^2)^2} &   \frac{-2x_0 x_1}{ (x_0^2 + x_1^2 + x_2^2)^2} &    \frac{-2x_0 x_2}{(x_0^2 + x_1^2 + x_2^2)^2} \\
 \frac{-2x_0 x_1}{(x_0^2 + x_1^2 + x_2^2)^2}  &  \frac{x_0^2 - x_1^2 + x_2^2}{ (x_0^2 + x_1^2 + x_2^2)^2}  &    \frac{-2x_1 x_2}{ (x_0^2 + x_1^2 + x_2^2)^2} \\
 \frac{-2x_0 x_2}{ (x_0^2 + x_1^2 + x_2^2)^2}  &   \frac{-2x_1 x_2}{ (x_0^2 + x_1^2 + x_2^2)^2}  &   \frac{x_0^2 + x_1^2 - x_2^2}{ (x_0^2 + x_1^2 + x_2^2)^2}
 \end{array}
 \right),
$$

$$
a_{-2}
 \left(
  \begin{array}{lll}
 \frac{2  x_0(-x_0^2 + 3x_1^2 + 3x_2^2)}{ (x_0^2 + x_1^2 + x_2^2)^3} &   \frac{-4 x_0^2x_1}{ (x_0^2 + x_1^2 + x_2^2)^3} &     \frac{-4 x_0^2x_2}{ (x_0^2 + x_1^2 + x_2^2)^3}  \\
  \frac{-4 x_0^2x_1}{ (x_0^2 + x_1^2 + x_2^2)^3}   &   \frac{2 x_0(x_0^2 - 3x_1^2 + x_2^2)}{ (x_0^2 + x_1^2 + x_2^2)^3}  &    \frac{-4 x_0x_1x_2}{ (x_0^2 + x_1^2 + x_2^2)^3} \\
 \frac{-4 x_0^2x_2}{ (x_0^2 + x_1^2 + x_2^2)^3} &   \frac{-4 x_0x_1x_2}{ (x_0^2 + x_1^2 + x_2^2)^3}  &   \frac{2 x_0(x_0^2 + x_1^2 - 3x_2^2)}{ (x_0^2 + x_1^2 + x_2^2)^3}
 \end{array}
 \right).
$$

The zero divergence condition
\begin{equation}
  \mathrm{div} \, \vec E = \frac{E_{\rho}}{\rho} = \frac{a_{-1}}{x_0^2 + x_1^2 + x_2^2} + \frac{2 a_{-2} x_0}{(x_0^2 + x_1^2 + x_2^2)^2} = 0
 \label{power zero divergence}
 \end{equation}
leads to an equation of a sphere of a radius $\frac{a_{-2}}{a_{-1}}$
 with center at the point $(-\frac{a_{-2}}{a_{-1}}, 0, 0)$:
$$
a_{-1}(x_0^2 + x_1^2 + x_2^2) + 2 a_{-2} x_0  = 0.
$$
 \end{ex}

 \begin{ex}
  Linear superposition of the reduced quaternionic exponential functions:
$\ F(x)= e^{-b_1 x} - e^{-b_2 x}$; $ \ b_1, b_2 \in \mathbb R$ $ \ (b_1, b_2 > 0).$ \\
We deal with a radially anti-holomorphic function  $\ \overline{F}(x) = \overline{e^{-b_1x}} + (-\overline{e^{-b_2x}}).$

The radially holomorphic potential in simply connected open domains $\Lambda\subset \mathbb R^3$  $(\rho > 0)$
 takes the form $G = - \frac{1}{b_1}e^{-b_1 x} + \frac{1}{b_2}e^{-b_2 x}$.

We get a meridional electrostatic field, where  $ \ \vec E = \vec
E^1 + \vec E^2$,
 $$
\vec E^1 =  \left( E_0^1, \frac{x_1}{\rho} E_{\rho}^1, \frac{x_2}{\rho} E_{\rho}^1 \right),  \ \ \ \ \  \vec E^2 =  \left( E_0^2, \frac{x_1}{\rho} E_{\rho}^2, \frac{x_2}{\rho} E_{\rho}^2 \right),
$$
 $$
E_0^1 =  e^{-b_1x_0} \cos (b_1 \rho),  \ \ \ \ \  E_{\rho}^1 = e^{-b_1x_0}  \sin (b_1 \rho),
 $$
 $$
E_0^2 =e^{-b_2x_0} \cos (b_2 \rho),  \ \ \ \ \  E_{\rho}^2 = e^{-b_2x_0}  \sin (b_2 \rho).
 $$

  The $EFG$ tensor $\mathbf{J}(\vec E^1)$ and the $EFG$ tensor $\mathbf{J}(\vec E^2)$ are  written as
$$
 \left(
  \begin{array}{lll}
-b_1 E_0^1 &   -\frac{b_1 x_1}{\rho} E_{\rho}^1  &    -\frac{b_1 x_2}{\rho} E_{\rho}^1 \\
 -\frac{b_1 x_1}{\rho} E_{\rho}^1  &  \left( \frac{b_1x_1^2}{ \rho^2} E_0^1 + \frac{x_2^2}{ \rho^3} E_{\rho}^1 \right) &   \left( \frac{b_1 x_1x_2}{\rho^2} E_0^1 -  \frac{x_1 x_2}{\rho^3} E_{\rho}^1 \right)  \\
 -\frac{b_1 x_2}{\rho} E_{\rho}^1 &   \left(  \frac{b_1 x_1x_2}{\rho^2} E_0^1 -  \frac{x_1 x_2}{\rho^3} E_{\rho}^1 \right)  &  \left( \frac{b_1x_2^2}{ \rho^2} E_0^1 + \frac{x_1^2}{ \rho^3} E_{\rho}^1 \right)
 \end{array}
 \right),
$$
$$
 \left(
  \begin{array}{lll}
b_2 E_0^2 &   \frac{b_2 x_1}{\rho} E_{\rho}^2  &    \frac{b_2 x_2}{\rho} E_{\rho}^2 \\
 \frac{b_2 x_1}{\rho} E_{\rho}^2  &  \left( -\frac{b_2x_1^2}{ \rho^2} E_0^2 - \frac{x_2^2}{ \rho^3} E_{\rho}^2 \right) &   \left( -\frac{b_2 x_1x_2}{\rho^2} E_0^2 + \frac{x_1 x_2}{\rho^3} E_{\rho}^2 \right)  \\
 \frac{b_2 x_2}{\rho} E_{\rho}^2  &   \left( -\frac{b_2 x_1x_2}{\rho^2} E_0^2 + \frac{x_1 x_2}{\rho^3} E_{\rho}^2 \right)  &  \left( -\frac{b_2x_2^2}{ \rho^2} E_0^2 - \frac{x_1^2}{ \rho^3} E_{\rho}^2 \right)
 \end{array}
 \right).
$$

The zero divergence condition
\begin{equation}
  \mathrm{div} \, \vec E =  e^{-b_1 x_0} \frac{\sin (b_1 \rho)}{\rho}
 - e^{-b_2 x_0} \frac{\sin (b_2 \rho)}{\rho} = 0
$$
implies that
$$
 e^{(b_2-b_1)x_0}\sin (b_1 \rho) - \sin (b_2 \rho) = 0.
 \label{exp zero divergence}
\end{equation}

In particular,  the Eq. $(\ref{exp zero divergence})$ under condition of $ b_2 = 2 b_1$ leads to equation of circular cylinders of increasing radius:
$$
 \sin (b_1 \rho) = 0,  \ \ \ \ \ \rho = \frac{\pi m}{b_1}, \ \ m= +1, +2,\ldots ,
$$
 and to equations described by separable variables $x_0$, $\rho$:
$$
   \cos (b_1 \rho) = \frac{e^{b_1 x_0}}{2},  \ \ \ \ \ x_0 = \frac{\ln [ 2 \cos (b_1 \rho)]}{b_1}.
$$
 \end{ex}

 \begin{rem}
The reduced quaternionic integral transforms of real-valued
originals within Fueter's construction  in $\mathbb R^3$
belong to joint class of solutions of the system $(H)$ and the
system  $(A_3)$ with variable coefficients \cite{Br:2003}.
 Their applications in different domains of mathematical physics were
explicitly demonstrated in 2011 \cite{BrW:2011}.
  \end{rem}

   \begin{defn}
 A real-valued function $\tilde{\eta} = \tilde{\eta}(\tau)$ of a real variable $\tau$
  is called an original real-valued function if
  \begin{enumerate}
  \item  the function $\tilde{\eta}(\tau)$ satisfies the H\"{o}lder's condition for each
    $\tau$ except points $\tau = \tau^1_{\tilde{\eta}},\tau^2_{\tilde{\eta}},\ldots$
  (there exists a finite quantity or zeros of such points for each finite interval), where
     $\tilde{\eta}(\tau)$ has gaps of the first kind,
  \item  for any $\tau<0$ $\tilde{\eta}(\tau) = 0$,
  \item for
  any $\tau >0$ there exist constants $B_{\tilde{\eta}}>0, \alpha_{\tilde{\eta}}\geq0$:
   $|\tilde{\eta}(\tau)| < B_{\tilde{\eta}} e^{\alpha_{\tilde{\eta}}\tau}$ .
  \end{enumerate}

 The H\"{o}lder's condition  for $ \tilde{\eta}(\tau)$ takes the following form: \\
 there exist constants $\ A_{\tilde{\eta}}>0,\
0<\lambda_{\tilde{\eta}}\leq1, \ \delta_{\tilde{\eta}}>0$ such that  \\
for each $\tau$ and $\delta$ $ \
|\tilde{\eta}(\tau+\delta)-\tilde{\eta}(\tau)| \leq
A_{\tilde{\eta}}|\delta|^{\lambda_{\tilde{\eta}}}$,
 where $|\delta|\leq\delta_{\tilde{\eta}}$.
\end{defn}

 \begin{defn}
Suppose that  the transform kernel takes the form $e^{-x \tau}$ and
$\rho > 0$. The reduced quaternionic integral transform of an
original real-valued function  $ \tilde{\eta}(\tau)$
$$
 F(x) := \mathfrak{LF} \{ \tilde{\eta}(\tau) ; x \}
 = \int_{0}^{\infty} \tilde{\eta}(\tau) e^{-x \tau} d\tau
$$
is called the one-sided  reduced quaternionic Laplace-Fueter transform of $ \tilde{\eta}(\tau)$.
 \end{defn}

 In the context of applications of the radially holomorphic potential
  we have to deal with radially anti-holomorphic functions
$$
  \overline{F}(x) = \int_{0}^{\infty} \tilde{\eta}(\tau) \overline{e^{-x \tau}} d\tau
= \int_{0}^{\infty} \tilde{\eta}(\tau) e^{-x_0\tau}[\cos(\rho\tau) +
I \sin(\rho\tau)]d\tau.
$$

Meridional models provided by the one-sided reduced quaternionic
Laplace-Fueter transform are given by relations:
$$
E_0 = \int_{0}^{\infty} \tilde{\eta}(\tau)
e^{-x_0\tau}\cos(\rho\tau) d\tau, \ \ \ \ \ E_{\rho} =
\int_{0}^{\infty} \tilde{\eta}(\tau) e^{-x_0\tau} \sin(\rho\tau)
d\tau.
$$

The zero divergence condition leads to a wide range of integral
equations depending on $ \tilde{\eta}(\tau)$:
\begin{equation}
   \int\limits_{0}^{\infty} \tilde{\eta}(\tau) e^{-x_0 \tau} \sin (\rho \tau) d\tau = 0.
\label{zero-div-one-Laplace}
 \end{equation}

The two-sided reduced quaternionic Laplace-Fueter transform may be
introduced, if values of original real-valued functions $\
\tilde{\eta} = \tilde{\eta}(\tau)$ do not vanish  identically for $
\tau < 0$
 (in complex analysis see details, e.g., \cite{PolBr}).

 \begin{defn}
Suppose that  the transform kernel takes the form $e^{-x \tau}$ and
$\rho > 0$. The reduced quaternionic integral transform of an
original real-valued function $ \tilde{\eta}(\tau)$ whose values do
not vanish identically for $ \tau < 0$
$$
 F(x) := \mathfrak{LF}_{-\infty}^{+\infty} \{ \tilde{\eta}(\tau) ; x \}
  = \int_{-\infty}^{\infty}\tilde{\eta}(\tau) e^{-x\tau} d\tau
$$
is called the two-sided reduced quaternionic Laplace-Fueter
transform of $ \tilde{\eta}(\tau)$.
 \end{defn}

 \begin{rem}
Euler's Gamma function of the reduced quaternionic argument
$\Gamma(x)$, where $ x_0>0$, was first introduced by the author in
2003 \cite{Br:2003}:
$$
  \Gamma(-x) := \mathfrak{LF}_{-\infty}^{+\infty} \{ e^{-e^{\tau}} ; x \}
= \int_{-\infty}^{\infty} e^{-e^{\tau}} e^{-x_0\tau}[\cos(\rho\tau)
- I \sin(\rho\tau)] d\tau.
$$
 \end{rem}

Meridional model provided by Euler's Gamma function of the reduced
quaternionic argument is given by relations:
$$
E_0  = \int_{-\infty}^{\infty} e^{-e^{\tau}}
e^{-x_0\tau}\cos(\rho\tau) d\tau, \ \ \ \ \ E_{\rho} =
\int_{-\infty}^{\infty} e^{-e^{\tau}} e^{-x_0\tau} \sin(\rho\tau)
d\tau.
$$

The zero divergence condition implies that
\begin{equation}
   \int\limits_{-\infty}^{\infty} e^{-e^{\tau}} e^{-x_0 \tau} \sin (\rho \tau) d\tau = 0.
\label{zero-div-Gamma}
 \end{equation}

 \begin{defn}
Suppose that  the transform kernel takes the form $\cos(x\tau)$ and
$\rho > 0$. The reduced quaternionic integral transform of an
original real-valued function $ \tilde{\eta}(\tau)$
$$
 F(x) := \mathfrak{FF}c \{ \tilde{\eta}(\tau) ; x \}
 = \int_{0}^{\infty} \tilde{\eta}(\tau) \cos(x\tau) d\tau = \frac{1}{2} \int_{0}^{\infty} \tilde{\eta}(\tau) ( e^{-Ix\tau} + e^{Ix\tau}) d\tau
$$
is called the reduced quaternionic Fourier-Fueter cosine transform of $ \tilde{\eta}(\tau)$.
 \end{defn}

Meridional models provided by the reduced quaternionic
Fourier-Fueter cosine transform are given by relations:
$$
E_0
 = \int_{0}^{\infty} \tilde{\eta}(\tau) \cosh(\rho\tau)\cos(x_0\tau)d\tau; \ \ \
E_{\rho}
 = \int_{0}^{\infty} \tilde{\eta}(\tau) \sinh(\rho\tau)\sin(x_0\tau)d\tau.
$$

The zero divergence condition leads to a wide range of integral
equations depending on $ \tilde{\eta}(\tau)$:
\begin{equation}
  \int_{0}^{\infty} \tilde{\eta}(\tau) \sinh(\rho\tau)\sin(x_0\tau)d\tau = 0.
\label{zero-div-cos-Fourier}
 \end{equation}

 \begin{rem}
Consider the following independent reduced quaternionic variable: $
\ y = Ix = -\rho + Ix_0$. The reduced quaternionic Fourier-Fueter
cosine transform of $ \tilde{\eta}(\tau)$ may be equivalently
represented by means of the one-sided reduced quaternionic
Laplace-Fueter transform:
$$
 \mathfrak{FF}c \{ \tilde{\eta}(\tau) ; x \}
 = \frac{1}{2} [ \mathfrak{LF} \{ \tilde{\eta}(\tau) ; y \} + \mathfrak{LF} \{ \tilde{\eta}(\tau) ; -y \}].
$$
 \end{rem}

 \begin{defn}
Suppose that  the transform kernel takes the form $\sin(x\tau)$ and
$\rho > 0$. The reduced quaternionic integral transform of an
original real-valued function $ \tilde{\eta}(\tau)$
$$
 F(x) := \mathfrak{FF}s \{ \tilde{\eta}(\tau) ; x \}
  = \int_{0}^{\infty} \tilde{\eta}(\tau) \sin(x\tau) d\tau = \frac{I}{2} \int_{0}^{\infty} \tilde{\eta}(\tau) ( e^{-Ix\tau} - e^{Ix\tau}) d\tau
$$
is called the reduced quaternionic Fourier-Fueter sine transform of $ \tilde{\eta}(\tau)$.
 \end{defn}

Meridional models provided by the reduced quaternionic
Fourier-Fueter sine transform are given by relations:
$$
E_0 = \int_{0}^{\infty} \tilde{\eta}(\tau)
\cosh(\rho\tau)\sin(x_0\tau)d\tau; \ \ \ E_{\rho} = -
\int_{0}^{\infty} \tilde{\eta}(\tau)
\sinh(\rho\tau)\cos(x_0\tau)d\tau.
$$

The zero divergence condition leads to a wide range of integral
equations depending on $ \tilde{\eta}(\tau)$:
\begin{equation}
  \int_{0}^{\infty} \tilde{\eta}(\tau) \sinh(\rho\tau)\cos(x_0\tau)d\tau = 0.
\label{zero-div-sin-Fourier}
 \end{equation}

 \begin{rem}
 The reduced quaternionic Fourier-Fueter sine transform of $ \tilde{\eta}(\tau)$ may be equivalently represented by
means of the one-sided reduced quaternionic Laplace-Fueter
transform:
$$
 \mathfrak{FF}s \{ \tilde{\eta}(\tau) ; x \}
 = \frac{I}{2} [ \mathfrak{LF} \{ \tilde{\eta}(\tau) ; y \} - \mathfrak{LF} \{ \tilde{\eta}(\tau) ; -y \}].
$$
 \end{rem}

 Specific properties of the reduced quaternionic Fourier-Fueter cosine and sine
transforms of original real-valued functions allow us, in contrast
to the reduced quaternionic Laplace-Fueter transform, to establish
integral representations of Bessel functions of the first kind of
integer order $n$ and the reduced quaternionic argument $x$.

 Rudiments of function theory in $\mathbb R^3$ developed by Leutwiler
(see, e.g., \cite{Leut:CV20,Leut:More95,Leut:Rud96}) allow to extend
 Bessel functions of the first kind of integer order $n$ from a
disk of radius $r$: $D_r = \{ (x_0, x_1): x_0^2 + x_1^2 < r^2 \} $
to the ball of radius $r$: $B_r^3 = \{ (x_0, x_1, x_2): x_0^2 +
x_1^2 + x_2^2 < r^2 \} $ by its reduced quaternionic power series
expansion with real coefficients within Fueter's construction in
$\mathbb R^3$ $(\ref{Fueter})$ (see, e.g., \cite{Watson:1944} in the
complex plane):
$$
J_{n}(x) = \sum_{m = 0}^\infty \frac{(-1)^m }{m!(n+m)!} \left(
\frac{x}{2} \right)^{n+2m}.
$$

Chebyshev polynomials of the first kind of even degree allow us to
establish integral representations of Bessel functions of the first
kind of even integer order and the reduced quaternionic argument:
$$
\frac{\pi}{2} (-1)^n J_{2n}(x)
  = \mathfrak{FF}c \{ \tilde{\eta}(\tau) ; x \}
 = \int_{0}^{1} \frac{ \cos (2n \arccos \tau )}{\sqrt{1 - \tau^2}} \cos(x\tau) d\tau,
 $$
where $ \ \tilde{\eta}(\tau) = \frac{T_{2n}(\tau)}{\sqrt{1 -\tau^2}}
 = \frac{\cos (2n \arccos \tau )}{\sqrt{1 - \tau^2}}$
 (see, e.g., \cite{Szeg:1967,Suetin,BatEr} in the complex plane).


Chebyshev polynomials of the first kind of odd degree allow us to
establish integral representations of Bessel functions of the first
kind of odd integer order and the reduced quaternionic argument:
$$
\frac{\pi}{2} (-1)^n J_{2n+1}(x)
  = \mathfrak{FF}s \{ \tilde{\eta}(\tau) ; x \}
 = \int_{0}^{1} \frac{ \cos [(2n+1) \arccos \tau ]}{\sqrt{1 - \tau^2}} \sin(x\tau) d\tau,
 $$
where $ \ \tilde{\eta}(\tau) = \frac{T_{2n+1}(\tau)}{\sqrt{1
-\tau^2}}
 = \frac{\cos [(2n+1) \arccos \tau ]}{\sqrt{1 - \tau^2}}$.

\begin{ex}
Bessel function of the first kind of order zero and the reduced
quaternionic argument is expressed as $ J_{0}(x) = \sum_{m =
0}^\infty \frac{(-1)^m }{(m!)^2} \left( \frac{x}{2} \right)^{2m}.$

 Original $\ \tilde{\eta}(\tau) = \frac{T_{0}(\tau)}{\sqrt{1 - \tau^2}}$ implies
that integral representation of $J_{0}(x)$ is expressed as
$$
 J_{0}(x)
  = \frac{2}{\pi} \mathfrak{FF}c \{ \tilde{\eta}(\tau) ; x \}
 = \frac{2}{\pi} \int_{0}^{1} \frac{1}{\sqrt{1 - \tau^2}}\cos(x\tau)
 d\tau.
$$

Meridional model provided by the reduced quaternionic Fourier-Fueter
cosine transform of original $\ \tilde{\eta}(\tau)$ is given by
relations:
$$
E_0 = \int_{0}^{1} \frac{\cosh(\rho\tau)}{\sqrt{1 - \tau^2}}
\cos(x_0\tau)d\tau; \ \ \ E_{\rho} = \int_{0}^{1}
\frac{\sinh(\rho\tau)}{\sqrt{1 - \tau^2}} \sin(x_0\tau)d\tau.
$$

The zero divergence condition leads to the following integral
equation:
$$
  \int_{0}^{1}
\frac{\sinh(\rho\tau)}{\sqrt{1 - \tau^2}} \sin(x_0\tau)d\tau = 0.
$$
 \end{ex}


Problems of applications of special radially holomorphic functions
in inhomogeneous media in $\mathbb R^3$ have not been studied in the
context of the theory of holomorphic functions in the plane and
$n$-dimensional space \cite{GuHaSp:2008,GuHaSp:2016}. Applications
of the radially holomorphic potential in electrostatics allow us to
make up for the gap.

\section{ Meridional Fields in Homogeneous Media and Harmonic Meridional Mappings of the Second Kind}

 Geometric  properties of the $EFG$ tensor $(\ref{EFG tensor-merid})$
 within meridional fields in case $\alpha = 0$
\begin{equation}
\left(
\begin{array}{lll}
 \left( -\frac{\partial{E_{\rho}}}{\partial{\rho}} - \frac{E_{\rho}}{\rho} \right) &  \frac{\partial{E_{\rho}}}{\partial{x_0}} \frac{x_1}{\rho} &
 \frac{\partial{E_{\rho}}}{\partial{x_0}} \frac{x_2}{\rho} \\
\frac{\partial{E_{\rho}}}{\partial{x_0}} \frac{x_1}{\rho}  & \left(
\frac{\partial{E_{\rho}}}{\partial{\rho}} \frac{x_1^2}{\rho^2}  +
\frac{E_{\rho}}{\rho} \frac{x_2^2}{\rho^2}\right)  &
 \left( \frac{\partial{E_{\rho}}}{\partial{\rho}}- \frac{E_{\rho}}{\rho}\right)  \frac{x_1 x_2}{\rho^2}  \\
\frac{\partial{E_{\rho}}}{\partial{x_0}} \frac{x_2}{\rho}  & \left(
\frac{\partial{E_{\rho}}}{\partial{\rho}}-
\frac{E_{\rho}}{\rho}\right)  \frac{x_1 x_2}{\rho^2}  & \left(
\frac{\partial{E_{\rho}}}{\partial{\rho}} \frac{x_2^2}{\rho^2} +
\frac{E_{\rho}}{\rho} \frac{x_1^2}{\rho^2}\right)
 \end{array}
\right) \label{EFG tensor-merid-0}
\end{equation}
 have not been studied.

 On the other hand, open problems in three-dimensional harmonic mappings of simply connected domains in the context
  of the theory of potential solenoid velocity fields
$\vec V = (V_0, V_1, V_2)$, where
$$
  \left\{
\begin{array}{l}
  \mathrm{div}\ { \vec V} = 0, \\
  \mathrm{curl}{\ \vec V} = 0,
 \end{array}
\right.
$$
were pointed out by Lavrentyev and Shabat in 1973
\cite{LavSh-Hydro}. Properties of the Jacobian matrix $\mathbf{J_{l
m}}(\vec V) = \frac{\partial{V_l}}{\partial{x_m}} \ $ $ ( l, m =
0,1,2)$ are difficult to treat in the general setting in contrast to
properties of the Jacobian matrix $\mathbf{J_{l m}}(\vec V) =
\frac{\partial{V_l}}{\partial{x_m}} \ $ $ ( l, m = 0,1)$ into the
framework of the theory of functions of a complex variable
\cite{LavSh}.

 An original approach to building special classes of three-dimensional harmonic mappings was
developed by Mel'nichenko in 1975 \cite{Melnic:1975} by means of
functions taking values in commutative associative algebras of the
third rank. As noted by Mel'nichenko and Plaksa in 1997
\cite{MelnicPlaksa:1997}, "it is impossible to select a special
class of axially symmetric potentials (quite interesting for
possible applications) in the collection of harmonic functions
constructed in \cite{Melnic:1975}".
 Potential fields with axial symmetry are of particular interest
to hydrodynamics problems in the context of GASPT (see, e.g.,
\cite{Weinstein:1948-flows,
Weinstein:1948-int,Weinstein:1953,Plaksa:2001,Plaksa:2003,Plaksa:2019}).

Let us look at properties of the $EFG$ tensor $(\ref{EFG
tensor-merid-0})$ taking into account that the system
 in the meridian half-plane $(\rho >0)$ $(\ref{Bryukhov-elec-meridional})$ is
expressed as
$$
\left\{
      \begin{array}{l}
         \rho \left( \frac{\partial{E_0}}{\partial{x_0}} + \frac{\partial{E_{\rho}}}{\partial{\rho}} \right) + E_{\rho} = 0,  \\
      \frac{\partial{E_0}}{\partial{\rho}} =
      \frac{\partial{E_{\rho}}}{\partial{x_0}},
     \end{array}
  \right.
$$
where $ E_0 = \frac{\partial{g}}{\partial{x_0}}, \ E_{\rho} =
\frac{\partial{g}}{\partial{\rho}}$.
 In accordance with the generalized Stokes-Beltrami system
$(\ref{generalized Stokes-Beltrami})$, generalized axially symmetric
potential $g = g(x_0, \rho)$ and the Stokes stream function $\hat{g}
= \hat{g}(x_0, \rho)$ satisfy equations
  $$
  \rho \left( \frac{\partial{^2}{g}}{\partial{x_0}^2} +  \frac{\partial {^2}{g}}{\partial{\rho}^2} \right)
  + \frac{\partial{g}}{\partial{\rho}} = 0, \quad \quad
  \rho \left( \frac{\partial{^2}{\hat{g}}}{\partial{x_0}^2} +  \frac{\partial {^2}{\hat{g}}}{\partial{\rho}^2} \right)
   - \frac{\partial{\hat{g}}}{\partial{\rho}} = 0.
$$

 The characteristic equation of the $EFG$ tensor $(\ref{EFG tensor-merid-0})$
  is written as incomplete cubic equation
\begin{equation}
 \lambda^3 +  II_{\mathbf{J}(\vec E)} \lambda  - III_{\mathbf{J}(\vec E)} = 0,
\label{character lambda}
\end{equation}
where

$ II_{\mathbf{J}(\vec E)} =  - \left[ \left( \frac{\partial{E_\rho}}{\partial{x_0}}  \right)^2 +  \left( \frac{\partial{E_{\rho}}}{\partial{\rho}} \right)^2 \right]
 - \frac{E_{\rho}}{\rho} \frac{\partial{E_{\rho}}}{\partial{\rho}} - \left( \frac{E_{\rho}}{ \rho} \right)^2,$

$   III_{\mathbf{J}(\vec E)} = - \frac{E_{\rho}}{\rho} \left[  \left( \frac{\partial{E_\rho}}{\partial{x_0}}  \right)^2 +  \left( \frac{\partial{E_{\rho}}}{\partial{\rho}} \right)^2 \right]
 - \left( \frac{E_{\rho}}{ \rho} \right)^2 \frac{\partial{E_{\rho}}}{\partial{\rho}}$.

 \begin{cor}
 Roots of the characteristic equation $(\ref{character lambda})$ are  given by two independent formulas
\begin{equation}
 \lambda_{0} = \frac{E_{\rho}}{\rho}, \quad
\lambda_{1, 2} = - \frac{E_{\rho}}{2 \rho}  \pm  \sqrt{\left(
\frac{E_{\rho}}{2 \rho} \right)^2 + \frac{E_{\rho}}{\rho}
\frac{\partial{E_{\rho}}}{\partial{\rho}} + \left(
\frac{\partial{E_{\rho}}}{\partial{x_0}}\right)^2 + \left(
\frac{\partial{E_{\rho}}}{\partial{\rho}} \right)^2}.
\label{incomplete lambda}
\end{equation}
\end{cor}

Exact formulas (\ref{incomplete lambda}) demonstrate explicitly
geometric properties of the $EFG$ tensor (\ref{EFG tensor-merid-0})
within meridional fields in homogeneous media.

An important concept of $M$(onogenic)-conformal mappings  $u = u_0 +
iu_1 + ju_2: \Lambda \rightarrow \mathbb{R}^3$ was introduced by
Malonek in 2000 \cite{Mal:2000} in the context of quaternionic
analysis in $\mathbb R^3$.
 New geometric properties of $M$-conformal mappings have been characterized by G\"{u}rlebeck and Morais
by means of the reduced quaternion-valued monogenic functions with
non-vanishing Jacobian determinant (see, e.g., \cite{GM:2010}).
Applications of mappings $ \overline{u} = u_0 - iu_1 - ju_2: \Lambda
\rightarrow \mathbb{R}^3$ in mathematical physics have not been
studied.

 This leads to the following definition.
 \begin{defn}
Let $\Lambda \subset \mathbb R^3$ be a simply connected open domain,
where \\ $x_1 \neq 0, x_2 \neq 0$. Assume that an exact solution
$(u_0, u_1, u_2)$ of the system $(R)$ satisfies the following
condition: $x_2 u_1 = x_1 u_2$ in $\Lambda$. Mapping $u = u_0 + iu_1
+ ju_2: \Lambda \rightarrow \mathbb{R}^3$ is called harmonic
meridional mapping of the first kind, and mapping $ \overline{u} =
u_0 - iu_1 - ju_2: \Lambda \rightarrow \mathbb{R}^3$ is called
harmonic meridional mapping of the second kind, respectively.
 \end{defn}

The principal invariants of harmonic meridional mappings of the
second kind coincide with the principal invariants of the $EFG$
tensor $(\ref{EFG tensor-merid-0})$.

 \begin{cor}
Suppose that $u_{\rho} = \frac{u_1}{x_1}\rho = \frac{u_2}{x_2}\rho$
($x_1 \neq 0$, $x_2 \neq 0$).
  The set of degenerate points of harmonic meridional mappings of the second kind
$ \overline{u} = u_0 - iu_1 - ju_2: \Lambda \rightarrow
\mathbb{R}^3$ is provided by two independent equations:
$$
{u_{\rho}}=0,\quad
\left(\frac{\partial{u_{\rho}}}{\partial{x_0}}\right)^2
+\left(\frac{\partial{u_{\rho}}}{\partial{\rho}}\right)^2
+ \frac{u_{\rho}}{\rho}\frac{\partial{u_{\rho}}}{\partial{\rho}}=0.
 $$
 \end{cor}

\begin{ex}
Consider a generalized axially symmetric potential in case $\alpha
=0$ using Bessel function of the first kind of order zero: $g(x_0,
\rho) = e^{\breve{\beta} x_0} J_{0}( \breve{\beta} \rho)$, where $
\rho > 0$.

$ E_0 = \frac{\partial{g}}{\partial{x_0}} = \breve{\beta}
e^{\breve{\beta} x_0} J_{0}( \breve{\beta} \rho), \ \ \ \ \ E_{\rho}
= \frac{\partial{g}}{\partial{\rho}} = e^{\breve{\beta} x_0} J'_{0}(
\breve{\beta} \rho)$.

The electric field strength is represented as

 $\vec E = (E_0, \frac{x_1}{\rho} E_{\rho}, \frac{x_2}{\rho} E_{\rho})$ $ = e^{\breve{\beta} x_0}
\left( \breve{\beta} J_{0}( \breve{\beta} \rho), \frac{x_1}{\rho}
J'_{0}( \breve{\beta} \rho), \frac{x_2}{\rho} J'_{0}( \breve{\beta}
\rho) \right).$

$ \frac{\partial{E_{\rho}}}{\partial{x_0}} =
\breve{\beta}e^{\breve{\beta} x_0} J'_{0}( \breve{\beta} \rho), \ \
\ \ \ \frac{\partial{E_{\rho}}}{\partial{\rho}} = e^{\breve{\beta}
x_0} J''_{0}( \breve{\beta} \rho). $

 The $EFG$ tensor $(\ref{EFG tensor-merid-0})$  is written as
$$
\tiny{ e^{\breve{\beta} x_0}
 \left(
  \begin{array}{lll}
\left[ -J''_{0}( \breve{\beta} \rho) - J'_{0}( \breve{\beta} \rho)
\frac{1}{ \rho} \right] & J'_{0}( \breve{\beta} \rho)
\frac{\breve{\beta}x_1}{\rho} & J'_{0}( \breve{\beta} \rho)
\frac{\breve{\beta}x_2}{\rho} \\ J'_{0}( \breve{\beta} \rho)
\frac{\breve{\beta}x_1}{\rho}  & \left[ J''_{0}( \breve{\beta} \rho)
\frac{x_1^2}{\rho^2} + J'_{0}( \breve{\beta}
\rho)\frac{x_2^2}{\rho^3} \right]  & \left[ J''_{0}( \breve{\beta}
\rho) - J'_{0}( \breve{\beta} \rho) \frac{1}{ \rho} \right]
\frac{x_1 x_2}{\rho^2} \\ J'_{0}( \breve{\beta} \rho)
\frac{\breve{\beta}x_2}{\rho}  & \left[ J''_{0}( \breve{\beta} \rho)
- J'_{0}( \breve{\beta} \rho) \frac{1}{ \rho} \right] \frac{x_1
x_2}{\rho^2}  &
 \left[ J''_{0}( \breve{\beta} \rho) \frac{x_2^2}{\rho^2} + J'_{0}( \breve{\beta} \rho)\frac{x_1^2}{\rho^3} \right]
 \end{array}
 \right) }
$$
 Roots of the characteristic equation $(\ref{character lambda})$ are given by formulas \\
$ \lambda_{0} =  e^{\breve{\beta} x_0} \frac{J'_{0}( \breve{\beta} \rho)}{\rho},$ \\
$ \lambda_{1, 2} = e^{\breve{\beta} x_0} ( -\frac{J'_{0}(
\breve{\beta} \rho)}{2 \rho} $
 $ \pm \sqrt{ (\breve{\beta}^2 + \frac{1}{4 \rho^2}) [J'_{0}( \breve{\beta} \rho)]^2 +  [J''_{0}( \breve{\beta} \rho)]^2 +  \frac{1}{ \rho} J'_{0}( \breve{\beta} \rho)J''_{0}( \breve{\beta} \rho)} ) $.

The set of degenerate points of the $EFG$ tensor $(\ref{EFG
tensor-merid-0})$ is provided by two independent equations:

 $ J'_{0}( \breve{\beta} \rho) =0,$
 $  \quad \quad \breve{\beta}^2 [J'_{0}( \breve{\beta} \rho)]^2 +  [J''_{0}( \breve{\beta} \rho)]^2 +  \frac{1}{ \rho} J'_{0}( \breve{\beta} \rho)J''_{0}( \breve{\beta} \rho)  =0.$
 \end{ex}

\section {Concluding Remarks}

Numerous mathematical problems of three-dimensional potential fields
in inhomogeneous media may be investigated by means of the system
$(\ref{isotropic-electrostatic-Maxwell-system-3})$.
In particular, in the context of the theory of conduction of heat
(see, e.g., \cite {Carslaw,Landis}) the coefficient $\phi=
\phi(x_0,x_1,x_2)$ and the scalar potential $h= h(x_0,x_1,x_2)$  may
be interpreted as the thermal conductivity $\kappa = \kappa(x_0,
x_1, x_2)$ and the steady state temperature $T = T(x_0,x_1,x_2)$,
respectively.

On the other hand, $\alpha$-axial-hyperbolic non-Euclidean
modification $(\ref{eq:A_3^alpha-system})$ of the system $(R)$
 leads to a family of Vekua type systems in cylindrical coordinates $(\ref{A_3^alpha system-meridional})$
within meridional models of potential fields in special
cylindrically layered media, where $\phi( \rho) = \rho^{-\alpha}$,
$\alpha >0$.

 Properties of potential fields in inhomogeneous anisotropic media raise the next issues for consideration.
Would contemporary problems of potential fields be characterized
using a generalized Riemannian modification of the system $(R)$?

A rich variety of analytic models  may be studied in the context of
the static Maxwell system in three dimensional inhomogeneous
anisotropic media described by a symmetric tensor $\mathbf{\Phi} =
(\phi_{lm})$ with $C^1$-components
 $ \phi_{lm} = \phi_{lm}(x_0,x_1,x_2)$ $(l, m = 0,1,2)$ and positive eigenvalues  $\mu_l = \mu_l(x_0, x_1, x_2)$  $(l = 0,1,2)$:
\begin{equation}
  \left\{
\begin{array}{l}
  \mathrm{div}\ { \mathbf{\Phi} \vec E} = 0, \\
  \mathrm{curl}{\ \vec E} = 0.
 \end{array}
\right.
\label{anisotropic-electrostatic-Maxwell-system-3}
\end{equation}
The  vector  $\vec D := \mathbf{\Phi} \vec E = (\sum\limits_{m =0}^2 \phi_{0m}E_m, \sum\limits_{m =0}^2 \phi_{1m}E_m, \sum\limits_{m =0}^2 \phi_{2m}E_m)$
 is known as the electrostatic induction (see, e.g., \cite{Sivukhin:3,SvetGub}).

 The electrostatic potential $ h= h(x_0,x_1,x_2)$ in simply connected open domains $\Lambda \subset \mathbb R^3$, where $ \vec E = \mathrm{grad} \ h$,
allows us to reduce $C^1$-solutions of the system $(\ref{anisotropic-electrostatic-Maxwell-system-3})$
 to $C^2$-solutions of the continuity equation (see, e.g., \cite{Chew,Sivukhin:3,SvetGub,Landis}):
\begin{equation}
  \mathrm{div}( \mathbf{\Phi} \ \mathrm{grad}{\ h}) = \sum\limits_{l =0}^2 \frac{\partial{}}{\partial{x_l}} \left(\sum\limits_{m =0}^2 \phi_{lm}\frac{\partial{h}}{\partial{x_m}}\right) = 0.
\label{anisotropic-divergence-form}
\end{equation}

\begin{rem}
 The system $(\ref{anisotropic-electrostatic-Maxwell-system-3})$ in the context of mathematical theory of multidimensional first order elliptic systems
  was interpreted by Auscher and Ros\'{e}n in 2012 as the generalized Cauchy-Riemann system \cite{Auscher:II}.
 \end{rem}

 Meanwhile, general class of $C^1$-solutions of the system
$(\ref{anisotropic-electrostatic-Maxwell-system-3})$ may be
equivalently represented as class of $C^1$-solutions of the
following first order elliptic system:
 \begin{equation}
 \left\{
     \begin{array}{l}
   \frac{\partial{(\phi_{00}u_0 - \phi_{01}u_1 - \phi_{02}u_2)}}{\partial{x_0}} +
  \frac{\partial{(\phi_{10}u_0 - \phi_{11}u_1 - \phi_{12}u_2)}}{\partial{x_1}} +
  \frac{\partial{(\phi_{20}u_0 - \phi_{21}u_1 - \phi_{22}u_2)}}{\partial{x_2}} = 0,  \\
      \frac{\partial{u_0}}{\partial{x_1}}=-\frac{\partial{u_1}}{\partial{x_0}}, \ \ \ \ \
      \frac{\partial{u_0}}{\partial{x_2}}=-\frac{\partial{u_2}}{\partial{x_0}}, \\
      \frac{\partial{u_1}}{\partial{x_2}}=\frac{\partial{u_2}}{\partial{x_1}},
     \end{array}
  \right.
\label{Bryukhov-anisotropic-3}
\end{equation}
 where $\vec E := (u_0, -u_1, -u_2)$.

Let us consider the Riemannian metric
\begin{equation}
 ds^2  = \sum_{l =0}^2 \sum_{m =0}^2 \check{g}_{l m}(x_0,x_1,x_2)dx_ldx_m,
\label{anisotropic-Riemannian-3}
\end{equation}
 such that metric tensor is written as $ \mathbf{\check{G}} = (\check{g}_{lm})$, $\mathrm{det}\mathbf{\check{G}} \neq 0,$
while contravariant tensor  is written as $ \mathbf{\check{G}}^{-1} = (\check{g}^{lm})$.

The Beltrami's second differential parameter (see, e.g., \cite{Eisenhart:Riem,Ahlfors:1981})
of the electrostatic potential $ h= h(x_0,x_1,x_2)$ takes the following form:
$$
  \frac{1}{\sqrt{\mathrm{det}\mathbf{\check{G}}}} \sum\limits_{l =0}^2 \frac{\partial{}}{\partial{x_l}}
\left( \sum\limits_{m =0}^2 \sqrt{\mathrm{det}\mathbf{\check{G}}} \ \check{g}^{lm}\frac{\partial{h}}{\partial{x_m}}\right) =0.
$$

The symmetric tensor $\mathbf{\Phi} = (\phi_{lm})$ is explicitly
constructed into the framework of the system
$(\ref{anisotropic-electrostatic-Maxwell-system-3})$: $
 \mathbf{\Phi} = \sqrt{\mathrm{det}\mathbf{\check{G}}} \ \mathbf{\check{G}}^{-1}.
\label{anisotropic-tensor-3}
$

 The system $(\ref{Bryukhov-anisotropic-3})$
may be considered as a generalized Riemannian modification of the system $(R)$
with respect to the Riemannian metric $(\ref{anisotropic-Riemannian-3})$.

In particular, the static Maxwell system $(\ref{anisotropic-electrostatic-Maxwell-system-3})$
in anisotropic media described by coefficients
 $\phi _{00}(x_2) = x_2^{-\alpha_{00}}$, $\phi _{11}(x_2) =  x_2^{-\alpha_{11}}$, $\phi _{22}(x_2) =  x_2^{-\alpha_{22}}$  $ (x_2 > 0)$,
where $\alpha_{00}, \alpha_{11}, \alpha_{22}  \in \mathbb{R}$, $ \
\phi _{01} = \phi _{02} = \phi _{12} = 0$, is expressed as
\begin{equation}
  \left\{
\begin{array}{l}
  x_2^{-\alpha_{00}} \frac{\partial{E_0}}{\partial{x_0}} + x_2^{-\alpha_{11}} \frac{\partial{E_1}}{\partial{x_1}} + x_2^{-\alpha_{22}} \frac{\partial{E_2}}{\partial{x_2}} - \alpha_{22} x_2^{-\alpha_{22} -1} E_2 = 0,  \\
  \mathrm{curl}{\ \vec E} = 0,
 \end{array}
\right.
 \label{alpha_1.2-anisotropic-electrostatic-Maxwell-system}
\end{equation}
 and the system $(\ref{Bryukhov-anisotropic-3})$ is simplified:
\begin{equation}
 \left\{
    \begin{array}{l}
 x_2^{-\alpha_{00}} \frac{\partial{u_0}}{\partial{x_0}} - x_2^{-\alpha_{11}} \frac{\partial{u_1}}{\partial{x_1}}
- x_2^{-\alpha_{22}} \frac{\partial{u_2}}{\partial{x_2}} + \alpha_{22} x_2^{-\alpha_{22} -1} u_2 = 0,  \\
      \frac{\partial{u_0}}{\partial{x_1}}=-\frac{\partial{u_1}}{\partial{x_0}},
      \ \ \ \frac{\partial{u_0}}{\partial{x_2}}=-\frac{\partial{u_2}}{\partial{x_0}}, \\
      \frac{\partial{u_1}}{\partial{x_2}}=\ \ \frac{\partial{u_2}}{\partial{x_1}}.
     \end{array}
  \right.
 \label{alpha_1.2-hyperbolic Riemannian modification}
\end{equation}
 The system $(\ref{alpha_1.2-hyperbolic Riemannian modification})$ may be interpreted as
$(\alpha_{00}, \alpha_{11}, \alpha_{22})$-hyperbolic Riemannian modification of the system $(R)$
 with respect to a Riemannian metric defined on the halfspace $\{x_2 > 0\}$ by formula:
$$
ds^2 = \frac{d{x_0}^2}{ x_2^{2\alpha_{00}}} + \frac{d{x_1}^2}{ x_2^{2\alpha_{11}}} + \frac{d{x_2}^2}{ x_2^{2\alpha_{22}}}.
$$

The continuity equation $(\ref{anisotropic-divergence-form})$ is written as
\begin{equation}
 x_2^{-\alpha_{00}} \frac{{\partial}^2{h}}{{\partial{x_0}}^2} +
 x_2^{-\alpha_{11}} \frac{{\partial}^2{h}}{{\partial{x_1}}^2} +
 x_2^{-\alpha_{22}} \frac{{\partial}^2{h}}{{\partial{x_2}}^2}
- \alpha_{22} x_2^{-\alpha_{22} -1} \frac{\partial{h}}{\partial{x_2}} =0.
 \label{alpha_1,2-hyperbolic-3}
\end{equation}

The Eq. $(\ref{alpha_1,2-hyperbolic-3})$ may be considered as a generalized anisotropic Weinstein equation
 in the context of the system $(\ref{alpha_1.2-anisotropic-electrostatic-Maxwell-system})$.


\end{document}